\documentclass[11pt,reqno]{amsart}
\usepackage{graphicx}
\usepackage{color}
\usepackage{mathrsfs,amsmath,amssymb,amsthm,amsfonts}

\usepackage{graphics,url,color,epsfig}
\usepackage{tikz}
\usepackage[colorlinks]{hyperref}
\AtBeginDocument{%
  \hypersetup{%
    linkcolor=blue,%
    citecolor=green,%
  }%
}
\usepackage{cleveref}
\usepackage{pgfplots}
\usepackage{multicol}

\makeatletter

\newcommand{\Rmnum}[1]{\expandafter\@slowromancap\romannumeral #1@}
\makeatother

\newtheorem{thm}{Theorem}[section]

\newtheorem{lemma}[thm]{Lemma}
\newtheorem{remark}[thm]{Remark}
\newtheorem{theorem}[thm]{Theorem}

\usepackage[numbers,sort&compress]{natbib}
\allowdisplaybreaks

\begin{document}

\author{Wenxiong Chen}
\address{Wenxiong Chen \newline\indent Department of Mathematical Sciences \newline\indent Yeshiva University \newline\indent New York, NY, 10033, USA}
\email{wchen@yu.edu}

\author{Congming Li}
\address{Congming Li \newline\indent School of Mathematical Sciences \newline\indent CMA-Shanghai, Shanghai Jiao Tong University \newline\indent Shanghai, 200240, P. R. China}
\email{congming.li@sjtu.edu.cn}

\author{Leyun Wu}
\address{Leyun Wu\newline\indent
School of Mathematics \newline\indent South China University of Technology \newline\indent Guangzhou, 510640, P. R. China\newline\indent
and 
\newline\indent The Institute of Mathmatical Sciences \& Department of Mathematics,\newline\indent The Chinese University of Hong Kong, \newline\indent
Shatin, N.T., Hong Kong,  P. R. China}
\email{leyunwu@scut.edu.cn}

\author{Zhouping Xin}
\address{Zhouping Xin
\newline\indent The Institute of Mathmatical Sciences \& Department of Mathematics,\newline\indent The Chinese University of Hong Kong, \newline\indent
Shatin, N.T., Hong Kong,  P. R. China}
\email{zpxin@ims.cuhk.edu.hk}

\title{Refined regularity for nonlocal elliptic equations and applications}

\begin{abstract}
In this paper, we establish refined regularity estimates
for nonnegative solutions to the fractional Poisson equation
$$
(-\Delta)^s u(x) =f(x),\,\, x\in B_1(0).
$$
Specifically, we have derived H\"{o}lder, Schauder, and Ln-Lipschitz regularity estimates for
any nonnegative solution $u,$
provided that  only the local  $L^\infty$ norm of $u$ is bounded.  These estimates stand in sharp contrast to the existing results where the global
$L^\infty$ norm of $u$ is required.  Our findings indicate that the local values of the solution $u$ and $f$
are sufficient to control the local values of higher order derivatives of $u$. Notably, this makes it possible to establish a priori estimates in unbounded domains
by using blowing up and re-scaling argument.

As applications, we derive singularity and decay estimates for solutions to some super-linear nonlocal problems in unbounded domains, and in particular, we obtain a priori estimates
for a family of fractional Lane-Emden type equations in $\mathbb{R}^n.$ This is achieved by adopting a different method using auxiliary functions, which is applicable to both local and nonlocal
problems.

\end{abstract}



\maketitle

\numberwithin{equation}{section}
\section{Introduction and main results}

In recent years, there has been a surge of interest in employing the fractional Laplacian to model a diverse array of physical phenomena characterized by long-range interactions. Notably, it finds applications in modeling anomalous diffusion, quasi-geostrophic flows, turbulence, water waves, molecular dynamics, and other phenomena (refer to \cite{BoG, CaV, Co, TZ} and the associated references). Specifically, the fractional Laplacian serves as the infinitesimal generator of a stable Lévy diffusion process (refer to  \cite{Be}), playing a crucial role in describing anomalous diffusions observed in various contexts such as plasmas, flame propagation, chemical reactions in liquids, and population dynamics.


The fractional Laplacian can be defined as
\begin{eqnarray}\label{eq1-1}
\begin{aligned}
(-\Delta)^s u(x)&=C_{n, s}PV \int_{\mathbb{R}^n}\frac{u(x)-u(y)}{|x-y|^{n+2s}}dy\\
&=C_{n, s}PV \lim_{\varepsilon \to 0}\int_{\mathbb{R}^n\backslash B_\varepsilon(x)}\frac{u(x)-u(y)}{|x-y|^{n+2s}}dy,
\end{aligned}
\end{eqnarray}
where $s$ is any real number between $0$ and $1,$ $PV$
stands for the Cauchy principal value, $B_{\varepsilon}(x)$ is the ball of radius $\varepsilon$ centered at $x,$
and $C_{n, s}$ is a
dimensional constant that depends on $n$ and $s$, precisely given by
$$
C_{n, s}=\left(\int_{\mathbb{R}^n}\frac{1-cos(\zeta_1)}{|\zeta|^{n+2s}}d \zeta\right)^{-1}.
$$
To ensure the integrability of
the right hand side of \eqref{eq1-1}, we require that $u\in C_{loc}^{1, 1}(\mathbb{R}^n)\cap \mathcal{L}_{2s}$, where
$$
\mathcal{L}_{2s}=\left\{u\in L_{loc}^1 \,\Big| \int_{\mathbb{R}^n}\frac{|u(x)|}{1+|x|^{n+2s}}dx<\infty \right\}
$$
endowed naturally with the norm
$$
\|u\|_{\mathcal{L}_{2s}}:= \int_{\mathbb{R}^n}\frac{|u(x)|}{1+|x|^{n+2s}}dx.
$$



The nonlocal nature of the fractional Laplacian makes it difficult to derive the regularity estimates by using the well-known
traditional approaches for local differential operators.  As a remedy, Caffarelli and Silvestre \cite{CS2007CPDE} introduced the extension method which turns nonlocal
equations into local ones in higher dimensions, so that a series of existing approaches on local equations can be applied, and thus it has become a powerful tool  in investigating equations involving the fractional Laplacian. However, one sometimes needs to assume extra  conditions such as  $s\geq \frac{1}{2}$ in proving symmetry of the fractional Lane-Emden equation (\cite{BCPS})  which may not be necessary if one approaches the problem directly.

 To address these challenges and explore further aspects of nonlocal equations involving the fractional Laplacian and other nonlocal elliptic operators, various powerful direct methods have been developed. These include the direct method of  moving planes and of moving spheres (\cite{ChenLiMa, CLL2017, ZL2022, DQ2018, WC2022, ChenLiZhang}), sliding methods (\cite{W2021, WC2020}), the method of scaling spheres (\cite{DQar}), and others (\cite{WC2022}). These methods have been widely applied to study nonlocal problems and a series of fruitful results have been obtained.


In this paper, by direct elaborate analysis instead of extension, we establish H\"{o}lder regularity, Schauder regularity, Ln-Lipschitz regularity and $C^k$ regularity for nonnegative classical solutions to the fractional Poisson equation
\begin{eqnarray}\label{eqFP}
	(-\Delta)^s u(x)=f(x), \,\, x \in B_1(0).
\end{eqnarray}

As important applications,
we derive  a priori estimates concerning possible singularities and decay for nonnegative solutions to nonlocal problems involving more general nonlinearities in an unbounded domain $\Omega$:
\begin{eqnarray}\label{eq1.7}
(-\Delta)^s u(x)=b(x) |\nabla u|^q(x)+f(x, u(x)),\,\, x\in \Omega,
\end{eqnarray}
where $0<q<\frac{2sp}{2s+p-1}.$

It  is noteworthy that the following interior
H\"{o}lder regularity
and Schauder regularity
for solutions to \eqref{eqFP} with $f \in L^\infty(B_1(0))$ and $f \in C^\alpha(\bar B_1(0))$ for all $0<\varepsilon < 2s$,  have been proved
\begin{eqnarray}\label{eqHC}
\|u\|_{C^{[2s-\varepsilon], \{2s-\varepsilon\}}(B_{1/2}(0))}\leq C(\|f\|_{L^\infty(B_1(0))}+\|u\|_{L^\infty(\mathbb{R}^n)}),
\end{eqnarray}
and
\begin{eqnarray}\label{eqSR}
\|u\|_{C^{[2s+\alpha], \{2s+\alpha\}}(B_{1/2}(0))}\leq
C\left(\|f\|_{C^\alpha(\bar B_1(0))} +\|u\|_{L^\infty(\mathbb{R}^n)}\right),
\end{eqnarray}
where $2s+\alpha \notin \mathbb{N},$  $[2s+\alpha]$ and $\{2s+\alpha\}$ are  the integer part and the fraction part of $2s+\alpha$ respectively,
see  \cite{ChenLiMa, DSV2017, Silvestre2007, Ros-Oton-Serra2016}  and also \cite{CS2009CPAM, CS2011ARMA, CS2011AM}
for fully nonlinear nonlocal equations.
Additionally, later results indicated that the right hand side of (\ref{eqSR}) can be  relaxed to
 $$
 C\left(\|f\|_{C^\alpha(\bar B_1(0))} +\|u\|_{\mathcal{L}_{2s}}\right).
 $$
 For more details, please refer to \cite[Theorem 12.2.4]{ChenLiMa}  and \cite[Theorem 1.3]{LW2021JDE}.

However, probably due to the non-local nature of the fractional Laplacian, it appears that these results are considerably weaker than those established for the
 classical Poisson equation
\begin{equation}\label{eq1-3}
-\Delta u(x)=f(x),\,\,	x\in B_1(0),
\end{equation}
for which one usually needs only $\|u\|_{L^\infty(B_1(0))}$  and $\|f\|_{C^\alpha(\bar B_1(0))}$ to bound
 $\|u\|_{C^{2, \alpha}(B_{1/2}(0))}.$

More importantly, the known a priori estimates, such as   \eqref{eqHC} and \eqref{eqSR} with its improvements are insufficient for applying the
powerful blowing-up  and re-scaling arguments to analyze the qualitative properties of solutions,  in particular, to establish a priori estimates for fractional equations on unbounded domains due to the requirement on the bonded-ness
of global norms of $u.$

As an illustration, let $\Omega$ be an unbounded domain in $\mathbb{R}^n$, and let us compare a simple problem of local nature
\begin{equation}
\left\{\begin{array}{ll}
- \Delta u(x) = u^p(x), & x \in \Omega\\
u(x) = 0, & x \in \partial \Omega
\end{array}
\right.
\label{1001}
\end{equation}
with it nonlocal counter part
\begin{equation}
\left \{\begin{array}{ll}
(- \Delta)^s u(x) = u^p(x), & x \in \Omega\\
u(x) = 0, & x \in \mathbb{R}^n \setminus \Omega,
\end{array}
\right.
\label{1002}
\end{equation}
where $0<s<1$, and $p$ is a subcritical exponent.

To obtain an a priori estimate of positive solutions, one would usually apply a typical blowing up and re-scaling argument briefly as follows.

If the solutions are not uniformly bounded, then there exist a sequence of solutions $\{u_k\}$ and a sequence of points $\{x^k\} \in \Omega$,
such that $u_k(x^k) \to \infty$. Notice that since $\Omega$ is unbounded, this $u_k(x^k)$ may not be the maximum of $u_k$ in $\Omega$.
Upon re-scaling, the new sequence of functions $v_k$ are bounded and satisfy the same equation on certain domains $B_k$.

{\em For the  equation in \eqref{1001} of local nature}, the bounded-ness of $\{v_k\}$ on $B_k$  is  {\em sufficient} to guarantee the bound on its higher norms, say $C^{2, \alpha}$ norm.  This leads to the convergence of  $\{v_k\}$ to a function $v$, which is a bounded solution of the same equation as \eqref{1001} either in the whole space or in a half space. By the known results on the nonexistence of solutions in these two spaces, one derives a contradiction and hence obtains the a priori estimate.

{\em Nevertheless, when dealing with nonlocal problem  \eqref{1002} }, the currently available regularity estimates, such as   \eqref{eqHC} and \eqref{eqSR}, indicate that the bounded-ness of $\{v_k\}$ on $B_k$ are {\em not sufficient} to guarantee the bound on its higher norms, they  also require $\{v_k\}$ {\em be bounded across the entire
space.} Such a condition cannot be met because there is no  information on the behavior of $v_k$ beyond $B_k$ during the re-scaling process. This posses  a substantial difficulty on employing blowing and re-scaling arguments for nonlocal equations on unbounded domains.
\medskip

The above arguments lead to the following natural question:

\medskip

Can one use only $\|u\|_{L^\infty(B_1(0))}$ and $\|f\|_{L^\infty(B_1(0))}$ ($\|f\|_{C^\alpha(\bar B_1(0))}$) to bound $\|u\|_{C^{[2s-\varepsilon], \{2s-\varepsilon\}}(B_{1/2}(0))}$
( $\|u\|_{C^{[2s+\alpha], \{2s-\alpha\}}(B_{1/2}(0))}$)
 in \eqref{eqHC} (\eqref{eqSR}) under certain conditions?
\medskip

This is one of the main problems we study  in this paper (see Theorem \ref{th1.1} and Theorem \ref{th1.2} below). By careful
analysis,
we answer this question affirmatively for nonnegative solutions and conclude that
$$
\|u\|_{C^{[2s-\varepsilon], \{2s-\varepsilon\}}(B_{1/2}(0))}\leq C(\|f\|_{L^\infty(B_1(0))}+\|u\|_{L^\infty(B_1(0))})
$$
for all $0<\varepsilon<2s,$
and
\begin{eqnarray}\label{fischauder}
\|u\|_{C^{[2s+\alpha], \{2s+\alpha\}}(B_{1/2}(0))}\leq
C\left(\|f\|_{C^\alpha(\bar B_1(0))} +\|u\|_{L^\infty(B_1(0))}\right)
\end{eqnarray}
for $2s+\alpha \notin \mathbb{N}.$

We now proceed to outline the main ingredients of our analysis and present our main results.
To derive the Schauder estimates for non-negative solutions of the fractional Poisson equation \eqref{eqFP}, we split the solution $u$ into two parts: the $s$-harmonic function in $B_1(0)$
$$
h(x):=u(x)-C_{n, s}\int_{B_1(0)}\frac{f(y)}{|x-y|^{n-2s}}dy
$$
and the remaining term
$$
w(x):=C_{n, s}\int_{B_1(0)}\frac{f(y)}{|x-y|^{n-2s}}dy.
$$
For  $w(x)$, one  expects  that $\|w\|_{C^{[2s+\alpha], \{2s+\alpha\}}(B_{1/2}(0))}$ can be controlled by  $\|f \|_{C^\alpha(B_1(0))}$. While for the $s$-harmonic function
$h(x)$, it has been proved in \cite{ChenLiMa} (see also \cite{LW2021JDE}) that $h(x)\in C^\infty(B_1(0))$ and can
be represented by
\begin{equation*}
h(x)=\int_{|y|>1} P(x, y)h(y)dy, \,\, \forall \; |x|<1,
\end{equation*}
where
\begin{eqnarray*}
P(x, y)=
\begin{cases}
\frac{\Gamma (n/2)\sin(\pi s)}{\pi^{\frac{n}{2}+1}}\left(\frac{1-|x|^2}{|y|^2-1}\right)^s\frac{1}{|x-y|^n}, &|y|>1,\\
0,& |y|\leq 1
\end{cases}
\end{eqnarray*}
is the so-called Poisson kernel.

By the definition of  $P(x, y)$ and some elaborate calculations, we find that
 the higher order derivatives of $P(x, y)$ can be controlled by itself. Then,
  $\|D^k h\|_{L^\infty(B_{1/2}(0))}$ can be bounded by $\|f\|_{L^\infty(B_1(0))}$ and  $\|u\|_{L^\infty(B_1(0))}$ due to the fact that $h=u-w$.
 This leads to better  Schauder estimates for nonnegative solutions, and our observations here may provide some useful  insights on the research
of local regularity for nonlocal operators.

Thus we have shown that
for  $2s+\alpha \not \in \mathbb{N}$,
$$
\|u\|_{C^{[2s+\alpha], \{2s+\alpha\}}(B_{1/2}(0))}\leq C\left(\|f\|_{C^\alpha(\bar B_1(0))}+\|u\|_{L^\infty(B_1(0))}\right).
$$



While for  $2s+\alpha \in \mathbb{N},$  two versions of estimates will be obtained:
\smallskip

i) for $f \in C^\alpha$,
 $\|u \|_{C^{2s+\alpha-1, \ln L}(B_{1/2}(0))}$ can be controlled by $\|u\|_{L^\infty(B_1(0))}$ and $\|f \|_{C^\alpha (\bar B_1(0))}$, which is the optimal Ln-Lipschitz
 regularity.



ii) for  $f \in C^{\alpha, Dini}$,  $\|u \|_{C^{2s+\alpha}(B_{1/2}(0))}$ can be estimated in terms of
$\|u\|_{L^\infty(B_1(0))}$ and $\|f \|_{C^{\alpha, Dini} (\bar B_1(0))}$.

Roughly speaking, the above regularity results imply that the local values of the solution $u$ and non-homogeneous term  $f$ are sufficient to control the higher order derivatives of $u$. These newly derived estimates are both novel and desirable.

Based on the above refined regularity estimates, we are able to establish singularity and decay estimates for solutions to equation \eqref{eq1.7} (see Theorem \ref{th1.4}).
To this end, one usually used the well-known ``doubling lemma" (\cite[Lemma 5.1]{PQS})  combining   with the re-scaling and blowing-up   to derive such kind of estimates.  For
example, Pol\'{a}\v{c}ik, Quittner and Souplet in \cite{PQS} established  decay and singularity  for the classical integer order Lane-Emden equation
$$
-\Delta u(x)=u^p(x)
$$
by using this method.

 In order to circumvent the use of ``the doubling lemma'' and make the proof more direct, we adopted an alternative approach involving the construction of auxiliary functions.
The idea is that if an estimate (in terms of the distance to $\partial \Omega$) such as
\eqref{eq1.8}
fails,
then by the contradiction argument, there exist sequences $ \Omega_k, \,\, u_k, \,\, x_k \in \Omega_k$ such that
$$u_k(x_k) \left(1+dist(x_k, \partial \Omega_k)^\frac{-2s}{p-1}\right)^{-1}>k \to \infty\,\,\, \mbox{ as }\,\,\, k\to \infty,$$
see \eqref{eq4-2}.
Then for each $k$, we construct an auxiliary function  $S_k(x)$ and find its maximum point denoted as  $a_k$ in a neighborhood of $x_k$. After appropriate rescaling (related to $a_k$), we can guarantee that the rescaled sequence $\{v_k\}$ is bounded in a neighborhood of $a_k$. In order to
 pass to the limit to obtain a bounded solution of a limiting problem in the whole space $\mathbb{R}^n$ to derive a contradiction with a Liouville theorem for the fractional Lane-Emden equation in $\mathbb{R}^n,$
we use the refined regularity
  results established in  Theorem \ref{th1.1} and Theorem \ref{th1.2} below.
Compared with previous results for fractional equations (see  \cite{LB2019JDE}),
we  do not use  ``the doubling lemma" and do not need to assume  $u \in \dot{H}^s(\mathbb{R}^n).$

This new method enables us
not only to obtain singularity and decay estimates for fractional order elliptic problems without using the doubling lemma, but also to deal effectively with both integer order elliptic and parabolic problems, as well as for fractional parabolic problems. As other by-products, we derive the boundedness of nonnegative solutions for a family of fractional Lane-Emden type equations in the whole space $\mathbb{R}^n$ and thus simplifies the problem when we study such kind of equations. These results can be extended to a family of
fractional Lane-Emden type systems.  It may shed some light on the fractional Lane-Emden conjecture
which  states that if
$$
\frac{1}{p+1}+\frac{1}{q+1}>\frac{n-2s}{n},
$$
then there are no nontrivial positive classical solutions of
\begin{eqnarray*}
\begin{cases}
(-\Delta)^s u(x)=v^p(x), & x\in \mathbb{R}^n,\\
(-\Delta)^s v(x)=u^q(x), & x\in \mathbb{R}^n.
\end{cases}
\end{eqnarray*}

We now present the main results in this paper in details.

We start with some standard notations.
 $C^{k, \ln L}(\bar \Omega)$ is the so called ``Ln-Lipshitz" spaces with
\begin{eqnarray*}
\|f\|_{C^{k, \ln L}(\bar \Omega)}=\|f\|_{C^{k}(\bar \Omega)}+\sum_{|\beta| = k}\sup_{x, y\in \Omega, x\neq y} \frac{|D^\beta f(x)-D^\beta f(y)|}{|x-y| \big|\ln \min(|x-y|, 1/2)\big|}<\infty.
\end{eqnarray*}

A function $f$ is said to be   Dini continuous (in $\Omega$) if $\omega_f$ satisfies
the Dini condition
$$
I_{\omega_f}:=\int_0^{\mbox{diam}(\Omega)}\frac{\omega_f(r)}{r}dt<+\infty,
$$
where
$$
\omega_f(r):=\sup \{|f(x)-f(y)| : x, y \in \Omega,\,\, |x-y|\leq r \}.
$$
Its norm is defined by
$$
\|f\|_{C^{Dini}(\bar \Omega)}=\|f\|_{C^0(\bar \Omega)}+I_{\omega_f}.
$$
 In addition, we say that $f$ is $C^{\alpha, Dini}$  (in $\Omega$) for
 $0<\alpha<1$ if $\omega_f$ satisfies
 $$
I_{\omega_f}^\alpha:=\int_0^{\mbox{diam}(\Omega)}\frac{\omega_f(r)}{r^{1+\alpha}}dr<+\infty
$$
with the norm
$$
\|f\|_{C^{\alpha, Dini}(\bar \Omega)}=\|f\|_{C^0(\bar \Omega)}+I_{\omega_f}^\alpha.
$$
It is obvious that Dini continuity is a stronger condition than continuity, and there exists a positive constant $C $ such that
$$\|f\|_{C^{\alpha}(\bar \Omega)}\leq C\|f\|_{C^{\alpha, Dini}(\bar \Omega)}.$$

The first main result is
 the H\"{o}lder regularity for fractional Poisson equations.

\begin{theorem}\label{th1.1}
Assume that $u\in C_{loc}^{1, 1}(\mathbb{R}^n)\cap \mathcal{L}_{2s}$ is a nonnegative solution of
\begin{equation}\label{eq1.1}
(-\Delta)^s u(x)=f(x),\,\, x \in B_1(0),
\end{equation}
where  $f(x) \in L^\infty(B_1(0)).$
 Then there exist some positive constants $C_1$ and $C_2$ such that
\begin{equation}\label{eq1.2}
\|u\|_{C^{[2s], \{2s\}}(B_{1/2}(0))}\leq C_1\left(\|f\|_{L^\infty(B_1(0))}+\|u\|_{L^\infty(B_1(0))}\right), \,\, \mbox{if}\,\, s\neq \frac{1}{2},
\end{equation}
and
\begin{equation}\label{eq1.3}
\|u\|_{C^{0, \ln L}(B_{1/2}(0))}\leq C_2\left(\|f\|_{L^\infty(B_1(0))}+\|u\|_{L^\infty(B_1(0))}\right), \,\, \mbox{if}\,\, s= \frac{1}{2},
\end{equation}
where $[2s]$ and $\{2s\}$ denote the integer part and the fraction part of $2s$ respectively.
\end{theorem}

\begin{remark}
Under the assumptions of Theorem \ref{th1.1}, for any $0<s<1$ and for all
$0<\varepsilon <2s,$ it follows from \eqref{eq1.2} and \eqref{eq1.3} that 	
$$
\|u\|_{C^{[2s-\varepsilon], \{2s-\varepsilon\}}(B_{1/2}(0))}\leq C\left(\|f\|_{L^\infty(B_1(0)}+\|u\|_{L^\infty(B_1(0))}\right).
$$
\end{remark}

If $f(x)\in C^\alpha(\bar B_1(0))\,(0<\alpha <1)$, then the following
 Schauder regularity for $2s+\alpha \not \in \mathbb{N}$ and Ln-Lipschitz regularity  for $2s+\alpha \in \mathbb{N}$ hold.

\begin{theorem}\label{th1.2}
Assume that $u\in C_{loc}^{1, 1}(\mathbb{R}^n)\cap \mathcal{L}_{2s}$ is a nonnegative solution of \eqref{eq1.1},  and $f(x)\in C^\alpha(\bar B_1(0))\,(0<\alpha <1)$.

If $2s+\alpha \notin \mathbb{N}$, then there exists a  constant $C_3>0$ such that
\begin{equation}\label{eq1.5}
\|u\|_{C^{[2s+\alpha], \{2s+\alpha\}}(B_{1/2}(0))}\leq C_3\left(\|f\|_{C^\alpha(\bar B_1(0))}+\|u\|_{L^\infty(B_1(0))}\right),
\end{equation}
where $[2s+\alpha]$  and $\{2s+\alpha\}$ are the integer part and the fraction part of $2s+\alpha$ respectively.

If $2s+\alpha \in \mathbb{N}$, then  there exists a  constant $C_4>0$ such that
\begin{eqnarray}\label{eq1.6}
\begin{aligned}
\|u\|_{C^{2s+\alpha-1, \ln L}(B_{1/2}(0))}
\leq  C_4\left(\|f\|_{C^\alpha(\bar B_1(0))}+\|u\|_{L^\infty(B_1(0))}\right).
\end{aligned}
\end{eqnarray}
\end{theorem}

\begin{remark}\label{re1.4}
It should be noted  that $\|u\|_{L^\infty(B_1(0))}$ can be replaced by 	
$\mathop{inf }\limits_{B_{3/4}(0)} u $ in  \eqref{eq1.2}, \eqref{eq1.3}, \eqref{eq1.5} and \eqref{eq1.6}. In fact,
first, it will follow from \eqref{eqi1} that $\|u\|_{L^\infty(B_1(0))}$
can be replaced by $\|u\|_{L^\infty(B_{1/2}(0))}$.
Second,
since
$h +\|w\|_{L^\infty(B_1(0))}+1 > 0 $ and $(-\Delta)^s (h +\|w\|_{L^\infty(B_1(0))}+1 ) =0$ in $B_1(0),$ then for any $x, y\in B_{3/4}(0),$ it holds that
\begin{eqnarray*}
	\begin{aligned}
h (x)+\|w\|_{L^\infty(B_1(0))}+1
\leq & \frac{(9/16-|x|^2)^s (3/4+|y|)^n}{(9/16-|y|^2)^s(3/4-|x|)^n} \left( h(y)+\|w\|_{L^\infty(B_1(0))}+1\right)\\
\leq &C \left( h(y)+\|w\|_{L^\infty(B_1(0))}+1\right),
	\end{aligned}
\end{eqnarray*}
which implies
$$\mathop{\sup}\limits_{B_{3/4}(0)} h \leq C \mathop{\inf}\limits_{B_{3/4}(0)} h.$$
It follows that
\begin{eqnarray*}
	\begin{aligned}
	\|u\|_{L^\infty(B_{1/2}(0))} &\leq \|h\|_{L^\infty(B_{3/4}(0))}+\|w\|_{L^\infty(B_{1}(0))}\\
	&\leq C \mathop{\inf}\limits_{B_{3/4}(0)} h +\|w\|_{L^\infty(B_{1}(0))}\\
	&\leq C  \mathop{\inf}\limits_{B_{3/4}(0)} u+ 2\|w\|_{L^\infty(B_{1}(0))}\\
	&\leq C \mathop{\inf}\limits_{B_{3/4}(0)} u+ 2\|f\|_{L^\infty(B_{1}(0))}.
	\end{aligned}
\end{eqnarray*}
Therefore, the conclusion is true.
\end{remark}

\begin{remark}
The estimates here are established by a direct method. It worth noting that
Jin, Li and Xiong (\cite[Theorem 2.11]{JLX2014}) addressed a similar problem by exploying an extension method and derived
\begin{eqnarray}\label{rfischauder}
\|u\|_{C^{[2s+\alpha], \{2s+\alpha\}}(B_{1/2}(0))}\leq
C\left(\|f\|_{C^\alpha(\bar B_1(0))} +\mathop{\inf}\limits_{B_{3/4}(0)} u\right)
\end{eqnarray}
for $2s+\alpha \notin \mathbb{N}$. As noted in the above Remark \ref{re1.4}, this estimate is essentially equivalent to our result in  (\ref{eq1.5}).
However, besides their basic condition that $u \in \mathcal{L}_{2s}$, due to the use of the extension method, they imposed the additional assumption that $u\in \dot{H}^s (\mathbb{R}^n).$ Clearly, this is a significantly stronger requirement than our condition that $u \in \mathcal{L}_{2s}$. There are many $\mathcal{L}_{2s}$ functions that do not belong to $\dot{H}^s (\mathbb{R}^n)$, for instance,
$$u(x) = |x| \mbox{ for } s > \frac{1}{2} \;\; \mbox{ and } \;\; u(x) = (x_n + 2)_+^\gamma \mbox{ for any }  2s > \gamma > s - \frac{n}{2}$$
are in $\mathcal{L}_{2s}$, however $\|u\|_{\dot{H}^s (\mathbb{R}^n)}= \infty$.
\end{remark}

\begin{remark}
 If there is no sign assumption on $u$, then the conclusions in
Theorem \ref{th1.1} and Theorem \ref{th1.2}  hold with $\|u\|_{L^\infty(B_1(0))}$ replaced by $\|u\|_{L^\infty(\mathbb{R}^n)},$ which have been proved  in the above mentioned references \cite{ChenLiMa, Silvestre2007, DSV2017, Ros-Oton-Serra2016, CS2009CPAM, CS2011ARMA, CS2011AM, LW2021JDE} and so on.
\end{remark}

Our next result concerns the $C^{2s+\alpha}$ regularity for nonnegative solutions of the fractional Poisson equation \eqref{eq1.1} under the condition
$f(x)\in C^{\alpha, Dini}(\bar B_1(0))\,(0<\alpha <1)$,  which seems to be optimal.

\begin{theorem}\label{th1.3}
Assume that $u\in C_{loc}^{1, 1}(\mathbb{R}^n)\cap \mathcal{L}_{2s}$ is a nonnegative solution of \eqref{eq1-3}. Let
$f(x)\in C^{\alpha, Dini}(\bar B_1(0))\,(0<\alpha <1)$.
If $2s+\alpha \in \mathbb{N}$, then  there exists a  constant $C_5>0$ such that
\begin{eqnarray*}
\begin{aligned}
\|u\|_{C^{2s+\alpha}(B_{1/2}(0))}
\leq  C_5\left(\|f\|_{C^{\alpha, Dini}(\bar B_1(0))}+\|u\|_{L^\infty(B_1(0))}\right).
\end{aligned}
\end{eqnarray*}
	\end{theorem}

\medskip

We now state the singularity and decay estimates for solutions to \eqref{eq1.7}  derived by constructing an auxiliary function combining with
rescaling, for which the following  conditions will be assumed:
 \smallskip

 (H1) $b(x):  \mathbb{R}^n \to \mathbb{R}$  is uniformly bounded and uniformly H\"{o}lder continuous.

 (H2) $f(x, t): \Omega \times [0, \infty) \to \mathbb{R}$ is uniformly H\"{o}lder continuous with respect to $x$ and continuous  with respect to $t$.

 (H3) $f(x, t) \leq C_0(1+t^p)$  uniformly for all $x$ in $\Omega$, and
 $$
 \mathop{\lim}\limits_{t\to \infty}\frac{f(x, t)}{t^p}=K(x), \,\, 1<p<\frac{n+2s}{n-2s},
 $$
 where $K(x)\in (0, \infty)$ is uniformly continuous and $ \mathop{\lim}\limits_{|x|\to \infty} K(x)=\bar C\in (0, \infty)$.

\begin{theorem}\label{th1.4}
Let $\Omega$ be an arbitrary domain of $\mathbb{R}^n$ and $u\in C_{loc}^{1, 1}(\Omega)\cap \mathcal{L}_{2s}$ be a nonnegative solution of \eqref{eq1.7} with $b(x)=0$. Suppose that the assumptions (H2) and (H3) hold. Then there
 exists a positive constant $C_6$ such that
\begin{eqnarray}\label{eq1.8}
u(x)\leq C_6\left(1+dist^{-\frac{2s}{p-1}}(x, \partial \Omega)\right),\,\, x\in \Omega.
\end{eqnarray}
More precisely, if $\Omega \neq \mathbb{R}^n,$ we have
\begin{eqnarray}\label{eq1.9}
u(x)\leq C_6 dist^{-\frac{2s}{p-1}}(x, \partial \Omega),\,\, x\in \Omega.
\end{eqnarray}
In particular,
if $\Omega$ is the whole space $\mathbb{R}^n,$ then
\begin{eqnarray}\label{eq1.10}
u(x)\leq C_6,\,\, x\in \mathbb{R}^n;
\end{eqnarray}
if $\Omega \neq \mathbb{R}^n $ is an exterior domain, i.e. $\Omega \supset \{x\in \mathbb{R}^n \mid |x|>R\},$
 then
\begin{eqnarray}\label{eq1.11}
u(x)\leq C_6 |x|^{-\frac{2s}{p-1}},\,\, |x|\geq 2R;
\end{eqnarray}
and if $\Omega$ is a punctured ball, i.e. $\Omega=B_R(0)\backslash\{0\}$ for some $R>0,$ then
\begin{eqnarray}\label{eq1.11-1}
u(x)\leq C_6 |x|^{-\frac{2s}{p-1}},\,\, 0<|x|<\frac{R}{2}.
\end{eqnarray}
\end{theorem}

\begin{remark}
As a typical example, \eqref{eq1.10} implies that all the nonnegative solutions to
\begin{eqnarray*}
(-\Delta)^s u(x)=u^p(x), \,\, x\in \mathbb{R}^n
\end{eqnarray*}
are bounded in the subcritical case when $1<p<\frac{n+2s}{n-2s}$.
Therefore, in this case,  the nonexistence of nontrivial nonnegative bounded solutions implies the nonexistence of
nontrivial nonnegative solutions, bounded or not.
This is also true for the Lane-Emden system under certain conditions, which may provide new
insights and ideas for completely solving the  Lane-Emden conjecture.
\eqref{eq1.9}, \eqref{eq1.11} and \eqref{eq1.11-1} provide a priori estimates for possible
singularities of local solutions to equation \eqref{eq1.7} and also decay estimates in the case of exterior domains.
 \end{remark}

If \eqref{eq1.7} contains $|\nabla u|$, we also derive  similar singularity and decay estimates.  In this case, we need to assume $s>\frac{1}{2}$ due to the appearance of the gradient term.

\begin{theorem}\label{th1.5}
Let $\Omega$ be an arbitrary domain in $\mathbb{R}^n$ and $u\in C_{loc}^{1, 1}(\Omega)\cap \mathcal{L}_{2s}$ is a nonnegative solution of \eqref{eq1.7}. Suppose that the assumptions (H1), (H2) and (H3) hold.

If $s>\frac{1}{2},$ then there exists a positive constant $C_7$ such that
$$
u(x)+|\nabla u|^{\frac{2s}{2s+p-1}}(x)\leq C_7\left(1+dist^{-\frac{2s}{p-1}}(x, \partial \Omega)\right),\,\, x\in \Omega.
$$
More precisely, if $\Omega \neq \mathbb{R}^n,$ we have
$$
u(x)+|\nabla u|^{\frac{2s}{2s+p-1}}(x)\leq C_7\, dist^{-\frac{2s}{p-1}}(x, \partial \Omega) ,\,\, x\in \Omega.
$$
In particular, if $\Omega$ is the whole space, i.e. $\Omega=\mathbb{R}^n,$ then
$$
u(x)+|\nabla u|^{\frac{2s}{2s+p-1}}(x)\leq C_7,\,\, x\in \mathbb{R}^n;
$$
if $\Omega \neq \mathbb{R}^n$ is an exterior domain, i.e. $\Omega \supset \{x\in \mathbb{R}^n \mid |x|>R\}$,   then
$$
u(x)+|\nabla u|^{\frac{2s}{2s+p-1}}(x)\leq C_7 |x|^{-\frac{2s}{p-1}},\,\, |x|\geq 2R;
$$
and if $\Omega$ is a punctured ball, i.e. $\Omega=B_R(0)\backslash \{0\}$ for some $R>0,$ then
$$
u(x)+|\nabla u|^{\frac{2s}{2s+p-1}}(x)\leq C_7 |x|^{-\frac{2s}{p-1}},\,\, 0<|x|\leq \frac{R}{2}.
$$
\end{theorem}

\begin{remark}
In the previous works on $L^\infty$-estimates for  solutions to fractional subcritical Lane-Emden equations (see \cite[Theorem 9.1.1]{ChenLiMa}),  it is often  assumed that $\Omega$
is bounded and each nonnegative solution is bounded in $\bar{\Omega}$.
 Then $\|u\|_{L^\infty(\Omega)}$   can be  bounded by a  uniform positive constant independent of $u.$ While
 Theorems \ref{th1.4} and \ref{th1.5} give
 $L^\infty$ estimate for  nonnegative solutions to the fractional subcritical Lane-Emden equation  in $\mathbb{R}^n$, where
there is no requirement on the boundedness of solutions in advance since
$\mathbb{R}^n$ is unbounded.
\end{remark}

The rest of the paper is organized as follows. In Section 2, the interior  H\"{o}lder continuity of nonnegative solutions to the factional Poisson equation  is obtained, which proves Theorem \ref{th1.1}. In Section 3, according to $2s+\alpha \not\in \mathbb{N}$ and $2s+\alpha \in \mathbb{N}$, we classify  completely  the regularity for nonnegative solutions to the fractional Poisson equation under the condition $f\in C^\alpha(\bar B_1(0))$, including Schauder regularity if $2s+\alpha \not\in \mathbb{N}$ and
Ln-Lipschitz regularity if $2s+\alpha \in \mathbb{N}$ (Theorem \ref{th1.2}).  In addition, we prove the $C^{2s+\alpha}$ regularity under the condition $f\in C^{\alpha, Dini}(\bar B_1(0))$ and $2s+\alpha \in \mathbb{N}$ (Theorem 1.3).
Section 4 is devoted to the proofs of singularity and  decay estimates without gradient term (Theorem \ref{th1.4}),  and
singularity and  decay estimates with gradient term (Theorem \ref{th1.5}).

\section{H\"{o}lder regularity}

In what follows, we will use $C$ and $C_i \, (i=1, 2, \cdots)$ to denote generic positive constants, whose values may differ from line to line.

In this section, we show the  H\"{o}lder regularity for any given  nonnegative solution $u$ to \eqref{eq1.1} by splitting  $u$ into two parts:
\begin{eqnarray*}
\begin{aligned}
u(x)&=\left(u(x)-C_{n, s}\int_{B_1(0)}\frac{f(y)}{|x-y|^{n-2s}}dy\right)+C_{n, s}\int_{B_1(0)}\frac{f(y)}{|x-y|^{n-2s}}dy\\
&:=h(x)+w(x).
\end{aligned}
\end{eqnarray*}

It can be proved that the first part
\begin{equation}\label{eq2.1}
h(x)=u(x)-C_{n, s}\int_{B_1(0)}\frac{f(y)}{|x-y|^{n-2s}}dy
\end{equation}
 is a $s$-harmonic function with
 $$
 (-\Delta)^s h(x)=0,\,\, x\in B_1(0).
 $$
The second part is denoted as
$$
w(x)=C_{n, s}\int_{B_1(0)}\frac{f(y)}{|x-y|^{n-2s}}dy.
$$
It is well-known that $w(x)$ is called the Newtonian potential of $f$ if $s=1$ and $n>2$. In this context, $w(x)$ is briefly called the potential of $f$ in the case that  $0<s<1$ and $n>2s$.

The following maximum principle will be needed in the proof.
\begin{lemma}\label{le2.1}\cite[Proposition 2.17]{Silvestre2007}
Let $\Omega$ be a bounded domain in $\mathbb{R}^n.$ Assume that $u$ is  a lower-semicontinuous function in $\bar \Omega$
and satisfies
\begin{eqnarray*}
\begin{cases}
(-\Delta)^s u(x)\geq 0, & x\in \Omega\\
u(x)\geq 0,& x\in \mathbb{R}^n \backslash \Omega.
\end{cases}
\end{eqnarray*}
in the sense of distribution.
Then $u(x) \geq 0$ in $\mathbb{R}^n.$
\end{lemma}

To derive the H\"{o}lder regularity for nonnegative solutions to the fractional Poisson equation \eqref{eq1.1}, we first estimate the $s$-harmonic
function $h(x)=u(x)-C_{n, s}\int_{B_1(0)}\frac{f(y)}{|x-y|^{n-2s}}dy$.

\begin{lemma}\label{le2.2} Let the assumptions in Theorem \ref{th1.1} hold and $h(x)$ is the  $s$-harmonic function defined in \eqref{eq2.1}. Then there exists a positive  constant $\tilde C_k$ such that for any $k=0, 1, \cdots,$  it holds that
$$
\|D^k h\|_{L^\infty(B_{1/2}(0))} \leq \tilde C_k \left(\|f\|_{L^\infty(B_1(0))}+\|u\|_{L^\infty(B_1(0))}\right).
$$
\end{lemma}

\begin{proof}
Since $h(x)\in \mathcal{L}_{2s},$ by Theorem 4.1.2 in \cite{ChenLiMa} or Lemma 4.1 in \cite{LLWX}, it holds that
$h(x)\in C^\infty(B_1(0)).$ Then it follows from Theorem 2.10 in \cite{Bucur} that $h(x)$ in $B_1(0)$ can be expressed in terms of an integral on $\mathbb{R}^n\backslash B_1(0)$ consisting of a Poisson kernel and $h$ itself:
\begin{equation}\label{eq2.3}
h(x)=\int_{|y|>1} P(x, y)h(y)dy, \,\, \forall |x|<1,
\end{equation}
where
\begin{eqnarray*}
P(x, y)=
\begin{cases}
\frac{\Gamma (n/2)\sin(\pi s)}{\pi^{\frac{n}{2}+1}}\left(\frac{1-|x|^2}{|y|^2-1}\right)^s\frac{1}{|x-y|^n}, &|y|>1,\\
0,& |y|\leq 1
\end{cases}
\end{eqnarray*}
is the so-called Poisson kernel.

Direct calculations, for any $|x|<\frac{1}{2}$ and $|y|>1,$ imply
\begin{eqnarray*}
\begin{aligned}
\frac{\partial P}{\partial x_i}(x, y)
=\left(\frac{-2s x_i}{1-|x|^2}+\frac{-n(x_i-y_i)}{|x-y|^2}\right)P(x, y).
\end{aligned}
\end{eqnarray*}
Then for any $|x|<\frac{1}{2}$ and $|y|>1,$ it holds that
\begin{eqnarray}\label{eq2.4}
\begin{aligned}
\left|\frac{\partial P}{\partial x_i}(x, y)\right|
\leq \left(\frac{2s |x|}{1-|x|^2}+\frac{n}{|x-y|}\right)P(x, y)
\leq C_1P(x, y).
\end{aligned}
\end{eqnarray}
For the second order derivatives of $h(x),$ for any  $|x|<\frac{1}{2}$ and $|y|>1,$ and each $i, j=1, \cdots, n$, one has
\begin{eqnarray*}
\begin{aligned}
\frac{\partial^2 P}{\partial x_i\partial x_j}(x, y)
=&\frac{\partial}{\partial x_j}\left(\frac{-2s x_i}{1-|x|^2}+\frac{-n(x_i-y_i)}{|x-y|^2}\right)P(x, y)\\
&+\left(\frac{-2s x_i}{1-|x|^2}+\frac{-n(x_i-y_i)}{|x-y|^2}\right)\frac{\partial P}{\partial x_j}(x, y).
\end{aligned}
\end{eqnarray*}
Then the smoothness of the function $\frac{-2s x_i}{1-|x|^2}+\frac{-n(x_i-y_i)}{|x-y|^2}$ and \eqref{eq2.4} imply
\begin{eqnarray*}
\begin{aligned}
\left|\frac{\partial^2 P}{\partial x_i\partial x_j}(x, y)\right|
\leq &\left|\frac{\partial}{\partial x_j}\left(\frac{-2s x_i}{1-|x|^2}+\frac{-n(x_i-y_i)}{|x-y|^2}\right)\right|P(x, y)\\
&+\left|\frac{-2s x_i}{1-|x|^2}+\frac{-n(x_i-y_i)}{|x-y|^2}\right|\frac{\partial P}{\partial x_j}(x, y)\\
\leq & C_2 P(x, y).
\end{aligned}
\end{eqnarray*}
By recursion, for any $k=1, 2, \cdots,$ we conclude that there exists a constant $C_k>0$ such that
\begin{eqnarray}\label{eq2.5}
\left|D^k P(x, y)\right|\leq C_k P(x, y), \,\, |x|<\frac{1}{2}\,\, \mbox{and}\,\,  |y|>1.
\end{eqnarray}

Due to \eqref{eq2.3} and  \eqref{eq2.5}, one can obtain
\begin{eqnarray}\label{eqhd}
\begin{aligned}
|D^k h(x)|=&\left|\int_{|y|>1} D^k P(x, y)h(y)dy\right|\\
=&\left|\int_{|y|>1} D^k P(x, y)(u(y)-w(y))dy\right|\\
\leq &\int_{|y|>1} \left|D^k P(x, y)\right|(u(y)+|w(y)|)dy\\
\leq & C_k \left(\int_{|y|>1}P(x, y)u(y)dy+\int_{|y|>1}P(x, y)|w(y)|dy\right)\\
:=& C_k (I_1+I_2),\,\, x\in B_{1/2}(0),
\end{aligned}
\end{eqnarray}
where the fact that $u$ is nonnegative has been used.

Note that
$$
|w(x)|=C_{n, s}\left|\int_{B_1(0)}\frac{f(y)}{|x-y|^{n-2s}}dy\right|\leq C\|f\|_{L^\infty(B_1(0))}, \,\, x\in \mathbb{R}^n \backslash B_{1}(0),
$$
hence
\begin{eqnarray}\label{eqi2}
I_2=\int_{|y|>1}P(x, y)|w(y)|dy\leq  C\|f\|_{L^\infty(B_1(0))}\int_{|y|>1}P(x, y)dy\leq  C\|f\|_{L^\infty(B_1(0))}.
\end{eqnarray}

To estimate $I_1,$ we define
\begin{eqnarray*}
v(x)=
\begin{cases}
\int_{|y|>1}P(x, y) u(y)dy, & x\in B_1(0),\\
u(x),& x\in \mathbb{R}^n\backslash B_1(0).
\end{cases}
\end{eqnarray*}
It follows from  the definition of $P(x, y)$ and  $u \in \mathcal{L}_{2s}$ that $v \in \mathcal{L}_{2s}.$

By virtue of \cite[Theorem 4.1.1]{ChenLiMa}, we obtain that $v(x)$ satisfies
\begin{eqnarray*}
\begin{cases}
(-\Delta)^s v(x)=0, & x\in B_1(0),\\
v(x)=u(x),& x\in \mathbb{R}^n\backslash B_1(0).
\end{cases}
\end{eqnarray*}

Recall that the function
$$
g_0(x)=\frac{2^{-2s}\Gamma(n/2)}{\Gamma((n+2s)/2)\Gamma(1+s)}(1-|x|^2_+)^s
$$
satisfies
$$
(-\Delta)^s g_0(x)=1.
$$
Set
$$
g(x)=\|f\|_{L^\infty(B_1(0))}g_0(x), \,\, x \in \mathbb{R}^n.
$$

Since $v-(u+g)\in \mathcal{L}_{2s}$ and satisfies
\begin{eqnarray*}
\begin{cases}
(-\Delta)^s (u+g-v)(x)=f(x)+\|f\|_{L^\infty(B_1(0))} \geq 0, & x\in B_1(0),\\
u(x)+g(x)-v(x)=0,& x\in \mathbb{R}^n\backslash B_1(0),
\end{cases}
\end{eqnarray*}
Lemma \ref{le2.1}  implies that
$$
v(x)\leq u(x)+g(x),\,\, x\in \mathbb{R}^n.
$$
It follows that
\begin{eqnarray}\label{eqi1}
I_1=\int_{|y|>1}P(x, y)u(y)dy\leq u(x)+g(x)\leq \|u\|_{L^\infty(B_1(0))}+C\|f\|_{L^\infty(B_1(0))},\,\, x \in B_{1/2}(0).
\end{eqnarray}
Collecting estimates \eqref{eqhd}--\eqref{eqi1} shows that
  for any $k=0, 1, \cdots ,$ there exists a positive  constant $\tilde C_k$ such that
  \begin{eqnarray*}
\begin{aligned}
\|D^k h\|_{L^\infty(B_{1/2}(0))} \leq  \tilde C_k\left(\|f\|_{L^\infty(B_1(0))}+\|u\|_{L^\infty(B_1(0))}\right).
\end{aligned}
\end{eqnarray*}
This completes the proof of Lemma \ref{le2.2}.
\end{proof}

Denote the fundamental solution of the fractional Laplacian $(-\Delta)^s$ in $\mathbb{R}^n$ by
$$
\Gamma(x, y):=\frac{C_{n, s}}{|x-y|^{n-2s}},\,\, n\geq 3.
$$
Then
$$
w(x)=\int_{B_1(0)} \Gamma (x, y) f(y)dy.
$$

\begin{lemma} \label{le3.1} If $f \in L^\infty(B_1(0))$ and  $2s>1,$ or
$f \in C^\alpha(\bar B_1(0))\,\, (0<\alpha<1)$ and  $2s+\alpha>1,$
 then $w \in C^1(B_1(0))$ satisfies
$$
\partial_i w(x)=\int_{B_1(0)}\partial_i \Gamma(x, y)\left(f(x)-f(y)\right)dy +f(x)\int_{\partial B_1(0)}\Gamma (x, y)\nu_i dS_y, \,\, i=1, \cdots, n,
$$
where $\nu_i$ is the $i$-th component of the unit outward normal vector of $\partial B_1(0).$
\end{lemma}

The proof of the lemma is standard  which is similar to the case $s=1$  in \cite{GT}.

\begin{proof}[Proof of Theorem \ref{th1.1}] Due to Lemma \ref{le2.2},
it suffices  to show
\begin{equation*}
\|w\|_{C^{[2s], \{2s\}}(B_{1/2}(0))}\leq C \left(\|f\|_{L^\infty(B_1(0))}+\|u\|_{L^\infty(B_1(0))}\right), \,\, \mbox{if}\,\, s\neq \frac{1}{2},
\end{equation*}
and
\begin{equation*}
\|w\|_{C^{0, \ln L}(B_{1/2}(0))}\leq C \left(\|f\|_{L^\infty(B_1(0))}+\|u\|_{L^\infty(B_1(0))}\right), \,\, \mbox{if}\,\, s= \frac{1}{2}.
\end{equation*}

For any $0<s<1,$ it is clear that
\begin{eqnarray}\label{eq2.5-2}
\|w\|_{L^\infty(B_{1/2}(0))}=C_{n, s}\sup_{B_{1/2}(0)}\int_{B_1(0)}\frac{f(y)}{|x-y|^{n-2s}}dy\leq C \|f\|_{L^\infty(B_1(0))}.
\end{eqnarray}

For any $x, \bar x \in B_{1/2}(0), $ let $\zeta=\frac{x+\bar x}{2}$ be the midpoint of $x$ and $\bar x$, and  $\delta=|x-\bar x|$ be the distance between $x$ and $\bar x.$ Then
\begin{eqnarray}\label{eq2.6}
\begin{aligned}
\left|w(x)-w(\bar x)\right|
=& C_{n, s}\left|\int_{B_1(0)}\left(\frac{1}{|x-y|^{n-2s}}-\frac{1}{|\bar x-y|^{n-2s}}\right)f(y)dy\right|\\
\leq &C_{n, s} \|f\|_{L^\infty(B_1(0))} \int_{B_2(\zeta)}\left|\frac{1}{|x-y|^{n-2s}}-\frac{1}{|\bar x-y|^{n-2s}}\right|dy\\
\leq &C_{n, s}\|f\|_{L^\infty(B_1(0))}  \left(\int_{B_\delta(\zeta)}\left|\frac{1}{|x-y|^{n-2s}}-\frac{1}{|\bar x-y|^{n-2s}}\right|dy\right.\\
&\left.+\int_{B_2(\zeta)\backslash B_\delta(\zeta)}\left|\frac{1}{|x-y|^{n-2s}}-\frac{1}{|\bar x-y|^{n-2s}}\right|dy\right).
\end{aligned}
\end{eqnarray}

Since $B_\delta(\zeta)\subset B_{\frac{3}{2}\delta}(x),$  if $s\neq \frac{1}{2},$ one has
\begin{eqnarray*}
\int_{B_\delta(\zeta)}\frac{1}{|x-y|^{n-2s}}dy\leq  \int_{B_{\frac{3}{2}\delta}(x)}\frac{1}{|x-y|^{n-2s}}dy
\leq C \int_0^{\frac{3}{2}\delta}  r^{2s-1}dr
\leq C \delta^{2s}.
\end{eqnarray*}
Thus for any  $s\in (0, 1),$ it holds that
\begin{eqnarray}\label{eq2.7}
\int_{B_\delta(\zeta)}\left|\frac{1}{|x-y|^{n-2s}}-\frac{1}{|\bar x-y|^{n-2s}}\right|dy
\leq \int_{B_\delta(\zeta)}\frac{1}{|x-y|^{n-2s}}dy+\int_{B_\delta(\zeta)}\frac{1}{|\bar x-y|^{n-2s}}dy
\leq C \delta^{2s}.
\end{eqnarray}
In addition, noting that
$$
|\xi-y|\geq |\zeta -y|-|\zeta-\xi|\geq \frac{|\zeta-y|}{2},\,\, y\in B_2(\zeta)\backslash B_\delta(\zeta),
$$
by the mean value theorem, if $0<s< \frac{1}{2},$ one can get
\begin{eqnarray}\label{eq2.8}
\begin{aligned}
&\int_{B_2(\zeta)\backslash B_\delta(\zeta)}\left|\frac{1}{|x-y|^{n-2s}}-\frac{1}{|\bar x-y|^{n-2s}}\right|dy\\
\leq & C\int_{B_2(\zeta)\backslash B_\delta(\zeta)}\frac{|x-\bar x|}{|\xi-y|^{n-2s+1}}dy \\
\leq & C |x-\bar x|\int_{B_2(\zeta)\backslash B_\delta(\zeta)} \frac{2}{|\zeta -y|^{n-2s+1}}dy\\
\leq & C |x-\bar x| \int_{\delta}^2 r^{2s-2}dr\\
\leq & C |x-\bar x|^{2s},
\end{aligned}
\end{eqnarray}
 where $\xi $ lies between $x$ and $\bar x.$

 If $s=\frac{1}{2},$ a similar argument as \eqref{eq2.8} yields
 \begin{eqnarray}\label{eq2.9}
 \begin{aligned}
 &\int_{B_2(\zeta)\backslash B_\delta(\zeta)}\left|\frac{1}{|x-y|^{n-2s}}-\frac{1}{|\bar x-y|^{n-2s}}\right|dy\\
\leq & C\int_{B_2(\zeta)\backslash B_\delta(\zeta)}\frac{|x-\bar x|}{|\xi-y|^{n}}dy \\
\leq & C |x-\bar x|\int_{B_2(\zeta)\backslash B_\delta(\zeta)} \frac{2}{|\zeta -y|^{n}}dy\\
\leq & C|x-\bar x|\big|\ln \min(|x-\bar x|, 1/2)\big|.
\end{aligned}
 \end{eqnarray}
 It follows from   \eqref{eq2.6}--\eqref{eq2.9}  that   for any $x, \bar x \in B_{1/2}(0), $
 \begin{eqnarray}\label{eq2.10}
 |w(x)-w(\bar x)|\leq C |x-\bar x|^{2s}\|f\|_{L^\infty(B_1(0))},\,\, \mbox{ if }\, 0<s<\frac{1}{2},
\end{eqnarray}
 and
 \begin{eqnarray}\label{eq2.11}
 |w(x)-w(\bar x)|\leq C |x-\bar x|\big|\ln \min(|x-\bar x|, 1/2)\big|\|f\|_{L^\infty(B_1(0))},\,\, \mbox{ if }\, s= \frac{1}{2}.
 \end{eqnarray}

Next we estimate $[\partial_i w]_{C^{2s-1}(B_{1/2}(0))},\, i=1, \cdots, n,$  by showing that
  \begin{eqnarray}\label{eq2.16}
 |\partial_i w(x)-\partial_i w(\bar x)|\leq C |x-\bar x|^{2s-1}\|f\|_{L^\infty(B_1(0))},\,\, \mbox{ if }\, \frac{1}{2}<s<1
\end{eqnarray}
 for any  $x, \bar x \in B_{1/2}(0) $.

 If  $\frac{1}{2}<s<1,$
 by Lemma \ref{le3.1}, for any $x \in B_{1/2}(0)$ and $i=1, \cdots, n,$ one has
\begin{eqnarray*}
\begin{aligned}
& |\partial_i w(x)|\\
=&\bigg|\int_{B_1(0)}\partial_i \Gamma(x-y)\left(f(y)-f(x)\right)dy +f(x)\int_{\partial B_1(0)} \Gamma(x-y) \nu_i dS_y \bigg|\\
\leq &\bigg|\int_{B_1(0)}\partial_i \Gamma(x-y)\left(f(y)-f(x)\right)dy\bigg| +\bigg|f(x)\int_{\partial B_1(0)} \Gamma(x-y) \nu_i dS_y \bigg|\\
\leq & C\|f\|_{L^\infty(B_1(0))}\int_{B_1(0)}\frac{1}{|x-y|^{n-2s+1}}dy +C\|f\|_{L^\infty(B_1(0))}\\
\leq &C \|f\|_{L^\infty(B_1(0))},
\end{aligned}
\end{eqnarray*}
which implies that
\begin{eqnarray}\label{eq2.5-3}
\|Dw\|_{L^\infty(B_{1/2}(0))}\leq  C \|f\|_{L^\infty(B_1(0))},\,\,  \mbox{if}\,\, \frac{1}{2}<s<1.
\end{eqnarray}

 If $\frac{1}{2}<s<1,$ by virtue of Lemma \ref{le3.1}, for any $x, \bar x \in B_{1/2}(0), $ it holds that
\begin{eqnarray*}
\begin{aligned}
&\partial_i w(x)-\partial_i w(\bar x)\\
=& \int_{B_1(0)} \partial_i \Gamma(x-y)\left(f(y)-f(x)\right) dy -\int_{B_1(0)}\partial_i\Gamma(\bar x-y)\left(f(y)-f(\bar x)\right) dy\\
&+f(x)\int_{\partial B_1(0)}\Gamma(x-y) \nu_i dS_y-f(\bar x)\int_{\partial B_1(0)}\Gamma(\bar x-y) \nu_i dS_y\\
:=& I_1+I_2,
\end{aligned}
\end{eqnarray*}
where
$$
I_1=\int_{B_1(0)} \partial_i \Gamma(x-y)\left(f(y)-f(x)\right) dy -\int_{B_1(0)}\partial_i\Gamma(\bar x-y)\left(f(y)-f(\bar x)\right) dy
$$
and
$$
I_2=f(x)\int_{\partial B_1(0)}\Gamma(x-y) \nu_i dS_y-f(\bar x)\int_{\partial B_1(0)}\Gamma(\bar x-y) \nu_i dS_y.
$$

For $I_1$,  one can  evaluate the integral on $B_\delta(\zeta)$ and
$B_1(0)\backslash B_\delta(\zeta)$ separately, where $\delta=|x-\bar x| $ and $\zeta=\frac{x+\bar x}{2}$.

For the integral $I_1$ on $B_\delta(\zeta),$ one has by direct estimates that
\begin{eqnarray}\label{eq2.17}
\begin{aligned}
&\left|\int_{B_\delta(\zeta)}\partial_i \Gamma(x-y) \left(f(y)-f(x)\right)dy\right|+\left|\int_{B_\delta(\zeta)}\partial_i \Gamma(x-y) \left(f(y)-f(\bar x)\right)dy\right|\\
\leq &  C \|f\|_{L^\infty(B_1(0))}\int_{B_\delta(\zeta)} \frac{1}{|x-y|^{n-2s+1}}dy\\
\leq &C  \|f\|_{L^\infty(B_1(0))}|x-\bar x|^{2s-1}.
\end{aligned}
\end{eqnarray}

The integral $I_1$ on $B_1(0)\backslash B_\delta(\zeta)$ can be estimated as follows
\begin{eqnarray}\label{eq2.18}
\begin{aligned}
&\int_{B_1(0) \backslash B_\delta(\zeta)}\partial_i \Gamma(x-y) \left(f(y)-f(x)\right)dy-\int_{B_1(0) \backslash B_\delta(\zeta)}\partial_i \Gamma(\bar x-y) \left(f(y)-f(\bar x)\right)dy\\
= & \int_{B_1(0) \backslash B_\delta(\zeta)}\left(\partial_i \Gamma(x-y)-\partial_i \Gamma(\bar x-y) \right)\left(f(y)-f(x)\right)dy\\
&+\int_{B_1(0) \backslash B_\delta(\zeta)}\partial_i \Gamma(\bar x-y) \left(f(\bar x)-f(x)\right)dy\\
:=&I_{11}+I_{12}.
\end{aligned}
\end{eqnarray}
The mean value theorem and the fact that  $f\in L^\infty( B_1(0))$ give
\begin{eqnarray}\label{eq2.19}
\begin{aligned}
|I_{11}|&=\bigg|\int_{B_1(0) \backslash B_\delta(\zeta)}\left(\partial_i \Gamma(x-y)-\partial_i \Gamma(\bar x-y) \right)\left(f(y)-f(x)\right)dy\bigg|\\
&\leq C\|f\|_{L^\infty(B_1(0))} \int_{B_1(0) \backslash B_\delta(\zeta)}\frac{|x-\bar x|}{|\xi-y|^{n-2s+2}}dy\\
&\leq C|f\|_{L^\infty(B_1(0))}    |x-\bar x|\int_{B_1(0) \backslash B_\delta(\zeta)}\frac{1}{|\zeta-y|^{n-2s+2}}dy\\
&\leq C|f\|_{L^\infty(B_1(0))} |x-\bar x|^{2s-1},
\end{aligned}
\end{eqnarray}
where $\xi$ is a point lying between $x$ and $\bar x$, and  one has used the fact that
$$
|x-y|\leq |x-\zeta|+|\zeta-y|\leq 2 |\zeta-y|, y\in B_1(0) \backslash B_\delta (\zeta),
$$
and
$$
|\xi-y|\geq |y-\zeta|-|\zeta -\xi|\geq |y-\zeta|-\frac{|y-\zeta |}{2}=\frac{|y-\zeta |}{2},\,\, y \in B_1(0)\backslash B_\delta(\zeta).
$$

  For $I_{12}$, one has by the divergence theorem that
 \begin{eqnarray*}
 I_{12}=\left(f(\bar x)-f(x)\right)\int_{\partial B_1(0)\cup \partial B_\delta(\zeta)} \Gamma (\bar x-y)\nu_i dS_y,
 \end{eqnarray*}
 which together with $I_2$ gives
 \begin{eqnarray}\label{eq2.20}
\begin{aligned}
& |I_{12}+I_2|\\
=&\bigg|\left(f(\bar x)-f(x)\right)\int_{\partial B_1(0)\cup \partial B_\delta(\zeta)} \Gamma (\bar x-y)\nu_i dS_y\\
&+f(x)\int_{\partial B_1(0)}\Gamma(x-y) \nu_i dS_y-f(\bar x)\int_{\partial B_1(0)}\Gamma(\bar x-y) \nu_i dS_y\bigg|\\
=&\bigg|\left(f(\bar x)-f(x)\right)\int_{\partial B_\delta(\zeta)} \Gamma (\bar x-y)\nu_i dS_y\\
&+f(x)\int_{\partial B_1(0)}\left(\Gamma(x-y)-\Gamma(\bar x-y)\right) \nu_i dS_y\bigg|\\
\leq & C[f]_{L^\infty(B_1(0))}|x-\bar x|^{2s}+|f(x)|\int_{\partial B_1(0)}\frac{|x-\bar x|}{|\xi-y|^{n-2s+1}}dS_y\\
\leq & C\|f\|_{L^\infty (B_1(0))}|x-\bar x|^{2s-1},
\end{aligned}
\end{eqnarray}
where $\xi$ lies between $x$ and $\bar x.$

Thus  \eqref{eq2.16}  follows from  \eqref{eq2.17}--\eqref{eq2.20}.

 By virtue of Lemma \ref{le2.2}, \eqref{eq2.5-2}, \eqref{eq2.5-3},
 \eqref{eq2.10}, \eqref{eq2.11} and \eqref{eq2.16}, we obtain \eqref{eq1.2} and \eqref{eq1.3}.
 This completes the proof of Theorem \ref{th1.1}.
\end{proof}

\section{Schauder regularity, Ln-Lipschitz regularity and \texorpdfstring{$C^{2s+\alpha}$}  ~ regularity}

In this section, we show the Schauder regularity and Ln-Lipschitz regularity for nonnegative solutions to \eqref{eq1.1}.
To this end, we need to
estimate $w(x):=\int_{B_1(0)} \Gamma (x, y) f(y)dy$ with $\Gamma(x, y):=\frac{C_{n, s}}{|x-y|^{n-2s}},\,\, n\geq 3.$

Assume that $f\in C^\alpha(\bar B_1(0)),$ which  may  be extended  to be $0$ outside of $B_1(0).$

\begin{lemma}\label{le3.2}
Let $w(x)$ be the potential of $f$ in $B_1(0).$ Then
\begin{eqnarray}\label{eq3.1}
\|w\|_{C^{[2s+\alpha], \{2s+\alpha\}}(B_{1/2}(0))}\leq C \|f\|_{C^\alpha(\bar B_1(0))},\,\, \mbox{ if }\,\, 2s+\alpha \notin \mathbb{N} ,
\end{eqnarray}
and
\begin{eqnarray}\label{eq3.2}
\|w\|_{C^{2s+\alpha-1, \ln L}(B_{1/2}(0))}\leq C \|f\|_{C^\alpha(\bar B_1(0))},\,\, \mbox{ if }\,\, 2s+\alpha \in \mathbb{N},
\end{eqnarray}
where $[2s+\alpha]$  and $\{2s+\alpha\}$ are the integer part and the fractional part of $2s+\alpha$ respectively.
\end{lemma}

\begin{proof}
If $2s+\alpha \notin \mathbb{N}, $  by a similar argument as Theorem 12.2.2  in \cite{ChenLiMa}, one can derive \eqref{eq3.1}.
Details of the proof are omitted here.

If $2s+\alpha \in \mathbb{N}$, to estimate $\|w\|_{C^{2s+\alpha-1, \ln L}(B_{1/2}(0))},$ we consider two
possible cases: $2s+\alpha =1$ and $2s+\alpha=2.$

Case i. $2s+\alpha =1.$

In this case, we show that
\begin{eqnarray}\label{eq3.3}
\|w\|_{C^{0, \ln L}(B_{1/2}(0))}\leq C \|f\|_{C^\alpha(\bar B_1(0))}.
\end{eqnarray}
By \eqref{eq2.5-2}, it suffices to show that for any $x,\, \bar x \in B_{1/2}(0),$
$$
|w(x)-w(\bar x)|\leq C [f]_{C^\alpha (\bar B_1(0))}|x-\bar x|\big|\ln \min(|x-\bar x|, 1/2)\big|,
$$
where $C$ is a constant independent of $x$ and $\bar x.$

Set $\zeta=\frac{x+\bar x}{2}.$ Then
$x, \bar x \in B_{1}(0)\subset B_2(\zeta).$

By symmetry,
$$
\int_{B_{2}(\zeta)}\left(\frac{1}{|x-y|^{n-2s}}-\frac{1}{|\bar x -y|^{n-2s}}\right)dy=0,
$$
it thus follows that
\begin{eqnarray}\label{eq3.4}
\begin{aligned}
&w(x)-w(\bar x)\\
=&C_{n, s}\int_{B_1(0)}\left(\frac{1}{|x-y|^{n-2s}}-\frac{1}{|\bar x-y|^{n-2s}}\right)f(y)dy\\
=&C_{n, s} \int_{B_2(\zeta)}\left(\frac{1}{|x-y|^{n-2s}}-\frac{1}{|\bar x-y|^{n-2s}}\right)f(y)dy\\
=&C_{n, s} \int_{B_2(\zeta)}\left(\frac{1}{|x-y|^{n-2s}}-\frac{1}{|\bar x-y|^{n-2s}}\right)(f(y)-f(x))dy,
\end{aligned}
\end{eqnarray}
due to  the fact $ \mbox{supp} f\subset B_1(0).$

Denote
$\delta=|x-\bar x|.$ We will estimate the  integrand  in \eqref{eq3.4}
on $B_\delta(\zeta)$ and $B_2(\zeta)\backslash B_\delta(\zeta)$
respectively.

By virtue of $f\in C^\alpha(\bar B_1(0))$, $B_\delta(\zeta)\subset B_{\frac{3}{2}\delta}(x)$ or $B_\delta(\zeta)\subset B_{\frac{3}{2}\delta}(\bar x)$, one has
\begin{eqnarray*}
\begin{aligned}
&\bigg|\int_{B_\delta(\zeta)}\frac{1}{|x-y|^{n-2s}}(f(y)-f(x))dy\bigg|\\
\leq& C [f]_{C^\alpha(\bar B_1(0))}\int_{B_\delta(\zeta)}\frac{1}{|x-y|^{n-1}}dy\\
\leq& C [f]_{C^\alpha(\bar B_1(0))}\int_{B_{\frac{3}{2}\delta}(x)}\frac{1}{|x-y|^{n-1}}dy\\
\leq & C [f]_{C^\alpha(\bar B_1(0))} |x-\bar x|.
\end{aligned}
\end{eqnarray*}
and
\begin{eqnarray*}
\begin{aligned}
&\bigg|\int_{B_\delta(\zeta)}\frac{1}{|\bar x-y|^{n-2s}}(f(y)-f(x))dy\bigg|\\
\leq& C [f]_{C^\alpha(\bar B_1(0))}\int_{B_\delta(\zeta)}\frac{|x-y|^\alpha}{|\bar x-y|^{n-2s}}dy\\
\leq& C [f]_{C^\alpha(\bar B_1(0))}\int_{B_\delta(\zeta)}\frac{|x-\bar x|^\alpha +|\bar x-y|^\alpha}{|\bar x-y|^{n-2s}}dy\\
\leq& C [f]_{C^\alpha(\bar B_1(0))}\left(|x-\bar x|^\alpha\int_{B_{\frac{3}{2}\delta}(\bar x)}\frac{1}{|\bar x-y|^{n-2s}}dy
+\int_{B_{\frac{3}{2}\delta}(\bar x)}\frac{1}{|\bar x-y|^{n-1}}dy\right)\\
\leq & C [f]_{C^\alpha(\bar B_1(0))} |x-\bar x|.
\end{aligned}
\end{eqnarray*}

Hence
\begin{eqnarray}\label{eq3.5}
\begin{aligned}
\left|\int_{B_\delta(\zeta)}\left(\frac{1}{|x-y|^{n-2s}}-\frac{1}{|\bar x-y|^{n-2s}}\right)(f(y)-f(x))dy\right|
\leq  C [f]_{C^\alpha(\bar B_1(0))} |x-\bar x|.
\end{aligned}
\end{eqnarray}
For any   $y \in B_2(\zeta)\backslash B_\delta(\zeta),$ since
$$
|x-y|\leq |x-\zeta|+|\zeta-y|\leq 2 |\zeta-y|.
$$
Then the mean value
theorem implies that
\begin{eqnarray}\label{eq3.6}
\begin{aligned}
&\left|\int_{B_2(\zeta)\backslash B_\delta(\zeta)}\left(\frac{1}{|x-y|^{n-2s}}-\frac{1}{|\bar x-y|^{n-2s}}\right)(f(y)-f(x))dy\right|\\
\leq &C \int_{B_2(\zeta)\backslash B_\delta(\zeta)} \frac{|x-\bar x|}{|\xi-y|^{n-2s+1}}[f]_{C^\alpha(B_1(0))}|x-y|^\alpha dy\\
\leq & C [f]_{C^\alpha(\bar B_1(0))} |x-\bar x| \int_{B_2(\zeta)\backslash B_\delta(\zeta)}\frac{(2 |\zeta-y|)^\alpha}{|\xi-y|^{n-2s+1}}dy\\
\leq & C [f]_{C^\alpha(\bar B_1(0))} |x-\bar x|\int_{B_2(\zeta)\backslash B_\delta(\zeta)}\frac{1}{|\zeta-y|^{n}}dy\\
\leq & C [f]_{C^\alpha (\bar B_1(0))}|x-\bar x|\big|\ln \min(|x-\bar x|, 1/2)\big|,
\end{aligned}
\end{eqnarray}
where $\xi$ is a point lying between $x$ and $\bar x$, and one has used the fact that
$$
|\xi-y|\geq |y-\zeta|-|\zeta -\xi|\geq |y-\zeta|-\frac{|y-\zeta |}{2}=\frac{|y-\zeta |}{2},\,\, y \in B_2(\zeta)\backslash B_\delta(\zeta).
$$
Collecting estimates  \eqref{eq3.4}--\eqref{eq3.6} yields
$$
|w(x)-w(\bar x)|\leq C [f]_{C^\alpha (\bar B_1(0))}|x-\bar x|\big|\ln \min(|x-\bar x|, 1/2)\big|,
$$
which shows that
\begin{eqnarray*}
\|w\|_{C^{0, \ln L}(B_{1/2}(0))}\leq C \|f\|_{C^\alpha(\bar B_1(0))}.
\end{eqnarray*}

Case ii. $2s+\alpha =2.$

In this case, we will estimate $\|\partial_i w(x)\|_{L^\infty(B_{1/2}(0))}$ and
$[w]_{C^{1, \ln L}(B_{1/2}(0))}$ to obtain
\begin{eqnarray}\label{eq3.6-1}
\|w\|_{C^{1, \ln L}(B_{1/2}(0))}\leq C\|f\|_{C^\alpha (\bar B_1(0))}|x-\bar x| \big|\ln \min(|x-\bar x|, 1/2)\big|.
\end{eqnarray}

By Lemma \ref{le3.1}, for any $x \in B_{1/2}(0),$ it holds that
\begin{eqnarray}\label{eq3.7}
\begin{aligned}
& |\partial_i w(x)|\\
=&\bigg|\int_{B_1(0)}\partial_i \Gamma(x-y)\left(f(y)-f(x)\right)dy +f(x)\int_{\partial B_1(0)} \Gamma(x-y) \nu_i dS_y \bigg|\\
\leq &\bigg|\int_{B_1(0)}\partial_i \Gamma(x-y)\left(f(y)-f(x)\right)dy\bigg| +\bigg|f(x)\int_{\partial B_1(0)} \Gamma(x-y) \nu_i dS_y \bigg|\\
\leq & C [f]_{C^\alpha(\bar B_1(0))}\int_{B_1(0)}\frac{1}{|x-y|^{n-1}}dy +C\|f\|_{L^\infty(B_1(0))}\\
\leq &C \|f\|_{C^\alpha(\bar B_1(0))}.
\end{aligned}
\end{eqnarray}

Next we estimate $[w]_{C^{1, \ln L}(B_{1/2}(0))}$ by  showing that
$$
|\partial_i w(x)-\partial_i w(\bar x)|\leq C\|f\|_{C^\alpha(\bar B_1(0))}|x-\bar x| \big|\ln \min(|x-\bar x|, 1/2)\big|, \,\,
\forall \, x, \bar x \in B_{1/2}(0).
$$

By virtue of Lemma \ref{le3.1}, for any $x, \bar x \in B_{1/2}(0), $ one has
\begin{eqnarray*}
\begin{aligned}
\partial_i w(x)-\partial_i w(\bar x)
:= I_1+I_2,
\end{aligned}
\end{eqnarray*}
where
$$
I_1=\int_{B_1(0)} \partial_i \Gamma(x-y)\left(f(y)-f(x)\right) dy -\int_{B_1(0)}\partial_i\Gamma(\bar x-y)\left(f(y)-f(\bar x)\right) dy
$$
and
$$
I_2=f(x)\int_{\partial B_1(0)}\Gamma(x-y) \nu_i dS_y-f(\bar x)\int_{\partial B_1(0)}\Gamma(\bar x-y) \nu_i dS_y.
$$

It can be estimated in a similar way as for the Case i by evaluating the integrands  in
 $B_\delta(\zeta)$ and
$B_1(0)\backslash B_\delta(\zeta)$ separately, where $\delta=|x-\bar x| $ and $\zeta=\frac{x+\bar x}{2}$.

On $B_\delta(\zeta)$,  note that
\begin{eqnarray*}
\begin{aligned}
&\left|\int_{B_\delta(\zeta)}\partial_i \Gamma(x-y) \left(f(y)-f(x)\right)dy\right|\\
\leq &  C [f]_{C^\alpha(\bar B_1(0))}\int_{B_\delta(\zeta)} \frac{1}{|x-y|^{n-1}}dy\\
\leq &C [f]_{C^\alpha(\bar B_1(0))}|x-\bar x|,
\end{aligned}
\end{eqnarray*}
and
\begin{eqnarray*}
\left|\int_{B_\delta(\zeta)}\partial_i \Gamma(\bar x-y) \left(f(y)-f(\bar x)\right)dy\right|\leq C [f]_{C^\alpha(\bar B_1(0))}|x-\bar x|.
\end{eqnarray*}
Then
\begin{eqnarray}\label{eq3.8}
\begin{aligned}
&\left|\int_{B_\delta(\zeta)}\partial_i \Gamma(x-y) \left(f(y)-f(x)\right)dy
-\int_{B_\delta(\zeta)}\partial_i \Gamma(\bar x-y) \left(f(y)-f(\bar x)\right)dy\right|\\
\leq &C [f]_{C^\alpha(\bar B_1(0))}|x-\bar x|.
\end{aligned}
\end{eqnarray}

On $B_1(0)\backslash B_\delta(\zeta)$, one has
\begin{eqnarray}\label{eq3.9}
\begin{aligned}
&\int_{B_1(0) \backslash B_\delta(\zeta)}\partial_i \Gamma(x-y) \left(f(y)-f(x)\right)dy-\int_{B_1(0) \backslash B_\delta(\zeta)}\partial_i \Gamma(\bar x-y) \left(f(y)-f(\bar x)\right)dy\\
= & \int_{B_1(0) \backslash B_\delta(\zeta)}\left(\partial_i \Gamma(x-y)-\partial_i \Gamma(\bar x-y) \right)\left(f(y)-f(x)\right)dy\\
&+\int_{B_1(0) \backslash B_\delta(\zeta)}\partial_i \Gamma(\bar x-y) \left(f(\bar x)-f(x)\right)dy\\
:=&I_{11}+I_{12}.
\end{aligned}
\end{eqnarray}
By the mean value theorem and the fact that $f\in C^\alpha(\bar B_1(0)),$ one can get
\begin{eqnarray}\label{eq3.10}
\begin{aligned}
|I_{11}|&=\bigg|\int_{B_1(0) \backslash B_\delta(\zeta)}\left(\partial_i \Gamma(x-y)-\partial_i \Gamma(\bar x-y) \right)\left(f(y)-f(x)\right)dy\bigg|\\
&\leq \int_{B_1(0) \backslash B_\delta(\zeta)}\frac{|x-\bar x|}{|\xi-y|^{n-2s+2}}[f]_{C^\alpha(B_1(0))}|x-y|^\alpha dy\\
&\leq C [f]_{C^\alpha(\bar B_1(0))}|x-\bar x|\int_{B_1(0) \backslash B_\delta(\zeta)}\frac{|\zeta-y|^{\alpha}}{|\zeta-y|^{n-2s+2}}dy\\
&\leq C [f]_{C^\alpha(\bar B_1(0))}|x-\bar x|\big|\ln \min(|x-\bar x|, 1/2)\big|,
\end{aligned}
\end{eqnarray}
where $\xi$ is a point lying between $x$ and $\bar x$, and one has used the fact that
$$
|x-y|\leq |x-\zeta|+|\zeta-y|\leq 2 |\zeta-y|, y\in B_1(0) \backslash B_\delta (\zeta),
$$
and
$$
|\xi-y|\geq |y-\zeta|-|\zeta -\xi|\geq |y-\zeta|-\frac{|y-\zeta |}{2}=\frac{|y-\zeta |}{2},\,\, y \in B_1(0)\backslash B_\delta(\zeta).
$$

 For the integral $I_{12}$, the divergence theorem yields
 \begin{eqnarray*}
 I_{12}=\left(f(\bar x)-f(x)\right)\int_{\partial B_1(0)\cup \partial B_\delta(\zeta)} \Gamma (\bar x-y)\nu_i dS_y,
 \end{eqnarray*}
 which, combined with  $I_2,$ gives
 \begin{eqnarray}\label{eq3.11}
\begin{aligned}
& |I_{12}+I_2|\\
=&\bigg|\left(f(\bar x)-f(x)\right)\int_{\partial B_1(0)\cup \partial B_\delta(\zeta)} \Gamma (\bar x-y)\nu_i dS_y\\
&+f(x)\int_{\partial B_1(0)}\Gamma(x-y) \nu_i dS_y-f(\bar x)\int_{\partial B_1(0)}\Gamma(\bar x-y) \nu_i dS_y\bigg|\\
=&\bigg|\left(f(\bar x)-f(x)\right)\int_{\partial B_\delta(\zeta)} \Gamma (\bar x-y)\nu_i dS_y\\
&+f(x)\int_{\partial B_1(0)}\left(\Gamma(x-y)-\Gamma(\bar x-y)\right) \nu_i dS_y\bigg|\\
\leq & C[f]_{C^\alpha (\bar B_1(0))}|x-\bar x|+|f(x)|\int_{\partial B_1(0)}\frac{|x-\bar x|}{|\xi-y|^{n-2s+1}}dS_y\\
\leq & C\|f\|_{C^\alpha (\bar B_1(0))}|x-\bar x|,
\end{aligned}
\end{eqnarray}
where $\xi$ lies between $x$ and $\bar x.$

It follows from  \eqref{eq3.7}--\eqref{eq3.11} that
\begin{eqnarray*}
\|w\|_{C^{1, \ln L}(B_{1/2}(0))}\leq C\|f\|_{C^\alpha (\bar B_1(0))}|x-\bar x| \big|\ln \min(|x-\bar x|, 1/2)\big|.
\end{eqnarray*}

Combining the two cases leads to  \eqref{eq3.2} by \eqref{eq3.3} and \eqref{eq3.6-1}.
This completes the proof of Lemma \ref{le3.2}.
 \end{proof}

\begin{proof}[Proof of Theorem \ref{th1.2}.]
To show the Schauder regularity for nonnegative solutions to the fractional Poisson equation  \eqref{eq1.1}  in the case $f \in C^\alpha(\bar B_1(0)) $ and $2s+\alpha \notin \mathbb{N},$ one needs only to combine the estimate for $w(x)$ in \eqref{eq3.1} and
with the estimate for $h(x)$ in Lemma \ref{le2.2}.

Then if  $2s+\alpha \notin \mathbb{N},$ it holds that
\begin{eqnarray*}
\begin{aligned}
&\|u\|_{C^{[2s+\alpha], \{2s+\alpha\}}(B_{1/2}(0))}\\
\leq &\|w\|_{C^{[2s+\alpha], \{2s+\alpha\}}(B_{1/2}(0))}+\|h\|_{C^{[2s+\alpha], \{2s+\alpha\}}(B_{1/2}(0))}\\
\leq &\|w\|_{C^{[2s+\alpha], \{2s+\alpha\}}(B_{1/2}(0))}+\|h\|_{C^{[2s+\alpha], \{2s+\alpha\}}(B_{1/2}(0))}\\
\leq & C\left(\|f\|_{C^\alpha(\bar B_1(0))}+\|u\|_{L^\infty(B_1(0))}\right),
\end{aligned}
\end{eqnarray*}
which shows \eqref{eq1.5}.

For the Ln-Lipschitz  regularity of nonnegative solutions to the fractional Poisson equation  \eqref{eq1.1}  in the case $f \in C^\alpha(\bar B_1(0)) $ and $2s+\alpha \in \mathbb{N},$ it suffices to combine the estimate for $w(x)$ in \eqref{eq3.2} and
with the estimate for $h(x)$ in Lemma \ref{le2.2}.

Therefore, if  $2s+\alpha \in \mathbb{N}$,  it follows from
\eqref{eq3.2} and Lemma \ref{le2.2} that
\begin{eqnarray*}
\begin{aligned}
&\|u\|_{C^{2s+\alpha-1, \ln L}(B_{1/2}(0))}\\
\leq &\|w\|_{C^{2s+\alpha-1, \ln L}(B_{1/2}(0))}+\|h\|_{C^{2s+\alpha-1, \ln L}(B_{1/2}(0))}\\
\leq &\|w\|_{C^{2s+\alpha-1, \ln L}(B_{1/2}(0))}+\|D^{2s+\alpha} h\|_{L^\infty(B_{1/2}(0))}\\
\leq & C\left(\|f\|_{C^\alpha(\bar B_1(0))}+\|u\|_{L^\infty(B_1(0))}\right).
\end{aligned}
\end{eqnarray*}
Therefore, \eqref{eq1.6} holds.
This completes the proof of Theorem \ref{th1.2}.
\end{proof}

\begin{proof}[Proof of Theorem \ref{th1.3}.]

If $2s+\alpha \in \mathbb{N} $  and $f \in C^{\alpha, Dini}(\bar B_1(0)),$ we can show the $C^{2s+\alpha}$ regularity for nonnegative solutions to the fractional Poisson equation \eqref{eq1.1}
by considering two cases: $2s+\alpha =1$ and $2s+\alpha =2.$

Case 1. $2s+\alpha =1.$

Then we claim that
$$
\|u\|_{C^{1}(B_{1/2}(0))}
\leq  C\left(\|f\|_{C^{\alpha, Dini}(\bar B_1(0))}+\|u\|_{L^\infty(B_1(0))}\right).
$$
Indeed, by  Lemma \ref{le2.2}, it suffices to prove
\begin{eqnarray}\label{eq3.13}
\|w\|_{C^{1}(B_{1/2}(0))}
\leq  C\|f\|_{C^{\alpha, Dini}(\bar B_1(0))}.
\end{eqnarray}

Let $\zeta$ be the midpoint of $x$ and $\bar x$, then $x, \bar x \in
B_1(0)\subset B_2(\zeta).$

Denote $\delta=|x-\bar x|.$ It follows  from \eqref{eq3.4} that
\begin{eqnarray}\label{eq3.14}
\begin{aligned}
&	w(x)-w(\bar x)\\
=&C_{n, s} \int_{B_2(\zeta)}\left(\frac{1}{|x-y|^{n-2s}}-\frac{1}{|\bar x-y|^{n-2s}}\right)(f(y)-f(x))dy\\
=&C_{n, s} \int_{B_\delta(\zeta)}\left(\frac{1}{|x-y|^{n-2s}}-\frac{1}{|\bar x-y|^{n-2s}}\right)(f(y)-f(x))dy\\
&+C_{n, s}\int_{B_2(\zeta)\backslash B_\delta(\zeta)}\left(\frac{1}{|x-y|^{n-2s}}-\frac{1}{|\bar x-y|^{n-2s}}\right)(f(y)-f(x))dy
\\
:=&C_{n, s}(I_1+I_2).
\end{aligned}
	\end{eqnarray}
 $I_1$ and $I_2$ can be treated respectively as follows.
	
By virtue of $f\in C^{\alpha, Dini}(\bar B_1(0))$ and $2s+\alpha =1,$ one can get
\begin{eqnarray*}
\begin{aligned}
&\left|\int_{B_\delta(\zeta)}\frac{1}{|x-y|^{n-2s}}\left(f(y)-f(x)\right)dy\right|\\
\leq & \left|\int_{B_{3\delta/2}(x)}\frac{1}{|x-y|^{n-2s-\alpha}} \|f\|_{C^{\alpha} (\bar B_1(0))} dy \right|\\
\leq & C \delta\|f\|_{C^{\alpha} (\bar B_1(0))} \\
\leq & C |x-\bar x| \|f\|_{C^{\alpha, Dini} (\bar B_1(0))},
\end{aligned}
	\end{eqnarray*}
and	
\begin{eqnarray*}
\begin{aligned}
&\left|\int_{B_\delta(\zeta)}\frac{1}{|\bar x-y|^{n-2s}}\left(f(y)-f(x)\right)dy\right|\\
\leq &\|f\|_{C^{\alpha} (\bar B_1(0))} \int_{B_\delta(\zeta)}\frac{|x-y|^\alpha}{|\bar x-y|^{n-2s}} dy\\
\leq &\|f\|_{C^{\alpha} (\bar B_1(0))}
\int_{B_\delta(\zeta)}\frac{|x-\bar x|^\alpha+|\bar x-y|^\alpha}{|\bar x-y|^{n-2s}} dy
\\
\leq &\|f\|_{C^{\alpha} (\bar B_1(0))}
\left(|x-\bar x|^\alpha
\int_{B_{3\delta/2}(\bar x)}\frac{1}{|\bar x-y|^{n-2s}} dy+
\int_{B_{3\delta/2}(\bar x)}\frac{1}{|\bar x-y|^{n-1}} dy
\right)
\\
\leq & C |x-\bar x| \|f\|_{C^{\alpha} (\bar B_1(0))} \\
\leq & C |x-\bar x| \|f\|_{C^{\alpha, Dini} (\bar B_1(0))}.
\end{aligned}
	\end{eqnarray*}

Hence,
\begin{eqnarray}\label{eq3.15}
|I_1|\leq C |x-\bar x| \|f\|_{C^{\alpha, Dini} (\bar B_1(0))}.
\end{eqnarray}

For  $I_2,$ since
$$
|x-y|\leq |x-\zeta|+|\zeta -y|\leq 2|\zeta -y|,\,\, x\in B_{1/2}(0),\,\, y\in B_2(\zeta)\backslash B_\delta(\zeta),
$$
one can use the mean value theorem to derive
\begin{eqnarray}\label{eq3.16}
\begin{aligned}
|I_2|=&\left|\int_{B_2(\zeta)\backslash B_\delta(\zeta)}\left(\frac{1}{|x-y|^{n-2s}}-\frac{1}{|\bar x-y|^{n-2s}}\right)(f(y)-f(x))dy\right|\\
\leq &C \int_{B_2(\zeta)\backslash B_\delta(\zeta)} \frac{|x-\bar x|}{|\xi-y|^{n-2s+1}}\omega_f(|x-y|)dy\\
\leq & C |x-\bar x| \int_{B_2(\zeta)\backslash B_\delta(\zeta)}\frac{\omega_f(2 |\zeta-y|)}{|\xi-y|^{n-2s+1}}dy\\
\leq & C |x-\bar x| \int_{B_2(\zeta)\backslash B_\delta(\zeta)}\frac{\omega_f(2 |\zeta-y|)}{|\zeta-y|^{n-2s+1}}dy\\
\leq & C |x-\bar x|\int_\delta^2 \frac{\omega_f(2r)}{r^{1+\alpha}}dr\\
\leq & C |x-\bar x| \|f\|_{C^{\alpha, Dini}(\bar B_1(0))},
\end{aligned}
\end{eqnarray}
where $\xi$ is a point lying between $x$ and $\bar x$, and one has used the fact that
$$
|\xi-y|\geq |y-\zeta|-|\zeta -\xi|\geq |y-\zeta|-\frac{|y-\zeta |}{2}=\frac{|y-\zeta |}{2},\,\, y \in B_2(\zeta)\backslash B_\delta(\zeta).
$$
Collecting \eqref{eq3.14}--\eqref{eq3.16} yields \eqref{eq3.13}.

Case 2. $2s+\alpha =2.$

Then it holds that
$$
\|u\|_{C^{2}(B_{1/2}(0))}
\leq  C\left(\|f\|_{C^{\alpha, Dini}(\bar B_1(0))}+\|u\|_{L^\infty(B_1(0))}\right),
$$
which follows from Lemma \ref{le2.2} and the following estimate
\begin{eqnarray}\label{eq3.17}
\|w\|_{C^{2}(B_{1/2}(0))}
\leq  C\|f\|_{C^{\alpha, Dini}(\bar B_1(0))}.
\end{eqnarray}

To prove \eqref{eq3.17},
by Lemma \ref{le3.1}, for any $x \in B_{1/2}(0),$ one has
\begin{eqnarray}\label{eq3.18}
\begin{aligned}
& |\partial_i w(x)|\\
=&\bigg|\int_{B_1(0)}\partial_i \Gamma(x-y)\left(f(y)-f(x)\right)dy +f(x)\int_{\partial B_1(0)} \Gamma(x-y) \nu_i dS_y \bigg|\\
\leq &\bigg|\int_{B_1(0)}\partial_i \Gamma(x-y)\omega_f(|x-y|)dy\bigg| +\bigg|f(x)\int_{\partial B_1(0)} \Gamma(x-y) \nu_i dS_y \bigg|\\
\leq & C \int_{B_1(0)}\frac{\omega_f(|x-y|)}{|x-y|^{n-2s+1}}dy +C\|f\|_{L^\infty(B_1(0))}\\
\leq &C \int_0^1\frac{\omega_f(r)}{r^{1+\alpha}}dy +C\|f\|_{L^\infty(B_1(0))}\\
\leq & C \|f\|_{C^{\alpha, Dini}(\bar B_1(0))}.
\end{aligned}
\end{eqnarray}

Next, the second order derivatives of $w$ can be estimated as
\begin{eqnarray}\label{eq3.18-1}
|\partial_i w(x)-\partial_i w(\bar x)|\leq C|x-\bar x| \|f\|_{C^{\alpha, Dini}(\bar B_1(0))}, \,\,
\forall \, x, \bar x \in B_{1/2}(0).
\end{eqnarray}
To this end, by Lemma \ref{le3.1}, for any $x, \bar x \in B_{1/2}(0), $  one has
\begin{eqnarray*}
\begin{aligned}
\partial_i w(x)-\partial_i w(\bar x)
:= J_1+J_2,
\end{aligned}
\end{eqnarray*}
where
$$
J_1=\int_{B_1(0)} \partial_i \Gamma(x-y)\left(f(y)-f(x)\right) dy -\int_{B_1(0)}\partial_i\Gamma(\bar x-y)\left(f(y)-f(\bar x)\right) dy
$$
and
$$
J_2=f(x)\int_{\partial B_1(0)}\Gamma(x-y) \nu_i dS_y-f(\bar x)\int_{\partial B_1(0)}\Gamma(\bar x-y) \nu_i dS_y.
$$

Decompose  $J_1$ into integrals on $B_\delta(\zeta)$ and $B_1(0) \backslash B_\delta(\zeta)$
and estimate each one as follows. First, it holds
\begin{eqnarray}\label{eq3.19}
\begin{aligned}
&\left|\int_{B_\delta(\zeta)} \partial_i \Gamma(x-y)\left(f(y)-f(x)\right) dy \right|	\\
\leq &C \int_{B_{3\delta/2}(x)}\frac{1}{|x-y|^{n-2s+1}}\omega_f(|x-y|)dy\\
\leq & C |x-\bar x|\int_0^{3\delta/2} \frac{\omega_f(r)}{r^{1+\alpha}}dr\\
\leq &C |x-\bar x| \|f\|_{C^{\alpha, Dini}(\bar B_1(0))},
\end{aligned}
\end{eqnarray}
similarly, one has
\begin{eqnarray}\label{eq3.20}
	\left|\int_{B_\delta(\zeta)} \partial_i \Gamma(\bar x-y)\left(f(y)-f(\bar x)\right) dy \right|\leq C |x-\bar x| \|f\|_{C^{\alpha, Dini}(\bar B_1(0))}.
\end{eqnarray}
Next, note that
\begin{eqnarray}\label{eq3.21}
\begin{aligned}
&\int_{B_1(0) \backslash B_\delta(\zeta)}\partial_i \Gamma(x-y) \left(f(y)-f(x)\right)dy-\int_{B_1(0) \backslash B_\delta(\zeta)}\partial_i \Gamma(\bar x-y) \left(f(y)-f(\bar x)\right)dy\\
= & \int_{B_1(0) \backslash B_\delta(\zeta)}\left(\partial_i \Gamma(x-y)-\partial_i \Gamma(\bar x-y) \right)\left(f(y)-f(x)\right)dy\\
&+\int_{B_1(0) \backslash B_\delta(\zeta)}\partial_i \Gamma(\bar x-y) \left(f(\bar x)-f(x)\right)dy\\
:=&J_{11}+J_{12}.
\end{aligned}
\end{eqnarray}
It follows from the mean value theorem and $f\in C^{\alpha, Dini}(\bar B_1(0))$ that
\begin{eqnarray}\label{eq3.22}
\begin{aligned}
|J_{11}|&=\bigg|\int_{B_1(0) \backslash B_\delta(\zeta)}\left(\partial_i \Gamma(x-y)-\partial_i \Gamma(\bar x-y) \right)\left(f(y)-f(x)\right)dy\bigg|\\
&\leq \int_{B_1(0) \backslash B_\delta(\zeta)}\frac{|x-\bar x|}{|\xi-y|^{n-2s+2}}\omega_f(|x-y|) dy\\
&\leq C |x-\bar x|\int_{B_1(0) \backslash B_\delta(\zeta)}\frac{\omega_f(2|\zeta-y|)}{|\zeta-y|^{n-2s+2}}dy\\
&\leq C |x-\bar x|\int_\delta^2 \frac{\omega_f(2r)}{r^{1+\alpha}}dr\\
&\leq C|x-\bar x|\|f\|_{C^{\alpha, Dini}(\bar B_1(0))},
\end{aligned}
\end{eqnarray}
where $\xi$ is a point lying between $x$ and $\bar x$, and one has used the facts that
$$
|x-y|\leq |x-\zeta|+|\zeta-y|\leq 2 |\zeta-y|,\,\, y\in B_1(0) \backslash B_\delta (\zeta),
$$
and
$$
|\xi-y|\geq |y-\zeta|-|\zeta -\xi|\geq |y-\zeta|-\frac{|y-\zeta |}{2}=\frac{|y-\zeta |}{2},\,\, y \in B_1(0)\backslash B_\delta(\zeta).
$$

For the integral $J_{12}$,    the divergence theorem implies
 \begin{eqnarray*}
 J_{12}=\left(f(\bar x)-f(x)\right)\int_{\partial B_1(0)\cup \partial B_\delta(\zeta)} \Gamma (\bar x-y)\nu_i dS_y.
 \end{eqnarray*}
Since
 $$
 |f(x)-f(\bar x)|\leq C |x-\bar x| \|f\|_{C^\alpha(\bar B_1(0))}\leq C |x-\bar x| \|f\|_{C^{\alpha, Dini} (\bar B_1(0))},
 $$
 we can estimate $J_{12}$  together  $J_2$  to obtain
 \begin{eqnarray}\label{eq3.24}
\begin{aligned}
& |J_{12}+J_2|\\
=&\bigg|\left(f(\bar x)-f(x)\right)\int_{\partial B_1(0)\cup \partial B_\delta(\zeta)} \Gamma (\bar x-y)\nu_i dS_y\\
&+f(x)\int_{\partial B_1(0)}\Gamma(x-y) \nu_i dS_y-f(\bar x)\int_{\partial B_1(0)}\Gamma(\bar x-y) \nu_i  dS_y\bigg|\\
=&\bigg|\left(f(\bar x)-f(x)\right)\int_{\partial B_\delta(\zeta)} \Gamma (\bar x-y)\nu_i dS_y\\
&+f(x)\int_{\partial B_1(0)}\left(\Gamma(x-y)-\Gamma(\bar x-y)\right) \nu_i  dS_y\bigg|\\
\leq & C|x-\bar x|^\alpha \|f\|_{C^{\alpha}(\bar B_1(0))}
|x-\bar x|^{2s-1} +|f(x)|\int_{\partial B_1(0)}\frac{|x-\bar x|}{|\xi -y|^{n-2s+1}}dS_y\\
\leq & C|x-\bar x| \|f\|_{C^{\alpha, Dini}(\bar B_1(0))},
\end{aligned}
\end{eqnarray}
where $\xi$ lies between $x$ and $\bar x,$ and one has used the fact $2s+\alpha =2.$

Collecting \eqref{eq3.19}--\eqref{eq3.24} leads to \eqref{eq3.18-1},
which together with \eqref{eq3.18} gives \eqref{eq3.17}.

This completes the proof of Theorem \ref{th1.3}.
\end{proof}

\section{A priori estimates}

\begin{proof} [Proof of Theorem \ref{th1.4}.]

By the assumptions,
$u$ sloves the  semi-linear equation
$$
(-\Delta)^s u(x)=f(x, u(x)),\,\, x\in \Omega,
$$
and we will show the following a priori estimates
$$
u(x)\leq C (1+ dist^{-\frac{2s}{p-1}}(x, \partial \Omega)), \,\, x\in \Omega
$$
by a contradiction argument.

Assume otherwise. Then, there exist sequences of $\Omega_k$, $u_k$, $x_k\in \Omega_k,$ such that
$u_k$ solves
\begin{eqnarray}\label{eq4-1}
\begin{cases}
(-\Delta)^s u_k(x)=f(x, u_k(x)),& x\in \Omega_k,\\
u_k(x)=0,& x\in \mathbb{R}^n\backslash \Omega_k,
\end{cases}
\end{eqnarray}
and
\begin{eqnarray}\label{eq4-2}
\frac{1}{2}(u_k(x_k))^{\frac{p-1}{2s}} d_k\geq
\frac{1}{2}(u_k(x_k))^{\frac{p-1}{2s}} (1+d_k^{-1})^{-1}
 >k\to \infty\,\, \mbox{as} \,\, k\to \infty,
\end{eqnarray}
where $d_k=dist (x_k, \partial \Omega_k).$

Define
$$
r_k:=2ku_k^{-\frac{p-1}{2s}}(x_k).
$$
Then \eqref{eq4-2} implies that $r_k<d_k$ and thus $B_{r_k}(x_k)\subset \Omega_k.$

Define a function
$$
S_k(x)=u_k(x)(r_k-|x-x_k|)^{\frac{2s}{p-1}},\,\, x\in B_{r_k}(x_k).
$$
Then there exists $a_k \in B_{r_k}(x_k)$ such that
$$
S_k(a_k)=\mathop{\max}\limits_{B_{r_k}(x_k)} S_k(x).
$$
It follows that
\begin{eqnarray}\label{eq4-3}
u_k(x)\leq u_k(a_k)\frac{(r_k-|a_k-x_k|)^{\frac{2s}{p-1}}}{(r_k-|x-x_k|)^{\frac{2s}{p-1}}},\,\, x\in B_{r_k}(x_k).
\end{eqnarray}

Set
$$
\lambda_k=(u_k(a_k))^{-\frac{p-1}{2s}}.
$$
Taking $x=x_k$ in \eqref{eq4-3} yields
 \begin{eqnarray}\label{eq4-4}
\frac{u_k(x_k)}{u_k(a_k)}\leq \frac{(r_k-|a_k-x_k|)^\frac{2s}{p-1}}{r_k^{\frac{2s}{p-1}}},
\end{eqnarray}
which implies  that
$$
u_k(x_k)\leq u_k(a_k).
$$
Using the definition of $r_k$ and $\lambda_k$, one gets from \eqref{eq4-4} that
 \begin{eqnarray}\label{eq4-5}
2k\lambda_k\leq r_k-|a_k-x_k|.
\end{eqnarray}

In addition, one can get
\begin{eqnarray}\label{eq4-6}
r_k-|a_k-x_k|\leq 2(r_k-|x-x_k|),\,\, x\in B_{k \lambda_k}(a_k).
\end{eqnarray}
Indeed, for any $x\in B_{k\lambda_k}(a_k)$, one has
$$
|x-x_k|\leq |x-a_k|+|a_k-x_k|\leq k\lambda_k +r_k-2k\lambda_k \leq r_k
$$
due to \eqref{eq4-5}. It follows that  $x\in B_{r_k}(x_k)$
and thus
$B_{k\lambda_k}(a_k)\subset B_{r_k}(x_k).$ Now for any $x\in B_{k\lambda_k}(a_k)$, using \eqref{eq4-5}, one can get
\begin{eqnarray*}
\begin{aligned}
	2(r_k-|x-x_k|)&\geq 2(r_k-(|x_k-a_k|+|x-a_k|))\\
	&\geq 2(r_k -(|x_k-a_k|+k\lambda_k))\\
	&\geq r_k -|x_k-a_k|+2k\lambda_k-2k\lambda_k\\
	&\geq r_k -|x_k-a_k|,
\end{aligned}
\end{eqnarray*}
therefore, \eqref{eq4-6} follows.

The combination of  \eqref{eq4-3} with \eqref{eq4-6} yields
\begin{eqnarray}\label{eq4-7}
u_k(x)\leq u_k(a_k) 2^\frac{2s}{p-1},\,\, x\in B_{k\lambda_k}(a_k).
\end{eqnarray}

 We now rescale the solutions as
 $$
   v_k(x)=\frac{1}{u_k(a_k)}u_k(\lambda_kx+a_k), \,\, x\in B_k(0).
 $$

By \eqref{eq4-1},
it is easy to check that the function $v_k(x)$ solves
\begin{eqnarray}\label{eq4-8}
(-\Delta)^s v_k(x)=
\lambda_k^\frac{2sp}{p-1} f(\lambda_kx+a_k, \lambda_k^{-\frac{2s}{p-1}}v_k(x))
:=F_k(x, v_k(x)),\,\, x\in B_k(0).
\end{eqnarray}
Moreover, it follows from the definition of $v_k$ and \eqref{eq4-7} that
\begin{eqnarray}\label{eq4-9}
v_k(0)=1
\end{eqnarray}
and
\begin{eqnarray}\label{eq4-10}
v_k(x)\leq 2^\frac{2s}{p-1},\,\, x\in B_k(0).
\end{eqnarray}
Then by \eqref{eq4-10} and the assumptions (H1)-(H3),  one can obtain
\begin{eqnarray*}
\begin{aligned}
F_k(x, v_k(x))
=&\lambda_k^\frac{2sp}{p-1} f(\lambda_kx+a_k, \lambda_k^{-\frac{2s}{p-1}}v_k(x))\\
\leq &C_0\lambda_k^\frac{2sp}{p-1}\left(1+\lambda_k^{-\frac{2sp}{p-1}}v_k^p(x)\right)\\
\leq &C_0\lambda_k^\frac{2sp}{p-1}+C_0 v_k^p(x)\\
\leq &C_1.
\end{aligned}
\end{eqnarray*}
Applying Theorem \ref{th1.1} to  \eqref{eq4-8} yields
\begin{eqnarray*}
\|v_k\|_{C^{[2s-\varepsilon], \{2s-\varepsilon\}}(B_{3k/4}(0))}\leq C_2
\end{eqnarray*}
for all $0<\varepsilon<2s.$
It follows that $F_k(x, v_k(x))$ is H\"{o}lder continuous
in $B_{3k/4}(0).$ Then by virtue of Theorem \ref{th1.2}, we conclude that there exist some positive constants
$\beta \in (0, 1)$ and $C_3$ independent of $k$ such that
\begin{eqnarray}\label{eq4-11}
\|v_k\|_{C^{[2s+\beta], \{2s+\beta\}}(B_{k/2}(0))}\leq C_3.
\end{eqnarray}

 Due to the above uniform regularity estimates, it then follows from the  Ascoli-Arzel\'{a} theorem that there exists
 a converging subsequence of $\{v_k\}$ (stilled denoted by $\{v_k\}$) that converges pointwisely in
 $\mathbb{R}^n$ to a function $v(x).$
Taking limit in \eqref{eq4-9} yields
\begin{eqnarray}\label{eq4-12}
v(0)=1.	
\end{eqnarray}
Moreover, $v_k$ converges in $C_{loc}^{[2s+\epsilon], \{2s+\epsilon\}}(\mathbb{R}^n)$ and hence  $(-\Delta)^s v_k$ converges point-wisely in $\mathbb{R}^n$.

Applying a result in \cite[Theorem 1.1]{DJXY2023} we obtain
\begin{eqnarray}\label{eq4-13}
\lim_{k \to \infty} (-\Delta)^s v_k(x)= (-\Delta)^s v(x)-b, \forall x \in \mathbb{R}^n
\end{eqnarray}
with
$$
b=C_{n, s} \lim_{R\to \infty} \lim_{k\to \infty} \int_{B_R^c} \frac{v_k(x)}{|x|^{n+2s}} dx \geq 0.
$$

%

In addition, for any fixed $x\in \mathbb{R}^n,$ if $\{\lambda_k x+a_k\}$ is bounded, we assume that $\lambda_k x+a_k \to \bar x$  by extracting a further
subsequence. Therefore, the assumption (H3) implies
\begin{eqnarray}\label{eq4-14}
\begin{aligned}
&\lambda_k^{\frac{2sp}{p-1}} f(\lambda_k x+a_k, \lambda_k^{-\frac{2s}{p-1}} v_k(x))\\
= & (v_k(x))^p\frac{f(\lambda_k x+a_k, \lambda_k^{-\frac{2s}{p-1}} v_k(x))}{\left(\lambda_k^{-\frac{2s}{p-1}} v_k(x)\right)^p}\\
\to& K(\bar x) v^p(x)\,\, \mbox{ as }\,\, k\to \infty,
\end{aligned}
\end{eqnarray}
which is  still valid in the case $|\lambda_k x+a_k| \to \infty$ with $|\bar x|= \infty$ by (H3).

Consequently,  it follows from \eqref{eq4-13} and \eqref{eq4-14} that $v$ is a solution of
\begin{eqnarray}\label{eq4-15}
(-\Delta)^s v(x)=K(\bar x) v^p(x)+b,\,\, x\in \mathbb{R}^n
\end{eqnarray}
with $b\geq 0.$

To proceed with the proof, we show that $b$ equals to $0.$

Let $$F_R(x) = \left\{\begin{array}{ll} K(\bar{x})v^p(x) + b, & x \in B_R(0)\\
0, & x \in B_R^C(0).
\end{array}
\right.
$$
$$
v_R(x)=C_{n, s}\int_{\mathbb{R}^n}  \frac{F_R(y)}{|x-y|^{n-2s}} dy.
$$
Then it is easy to see that
$$
\left\{\begin{array}{ll} (-\Delta)^s v_R(x)= F_R(x), \;\;  x \in \mathbb{R}^n\\
\mathop{\lim}\limits_{|x| \to \infty} v_R(x)=0.
\end{array}
\right.
$$

Denote
$$
w_R(x)=v(x)-v_R(x), \,\, x\in \mathbb{R}^n.
$$
Then
\begin{eqnarray*}
\begin{cases}
(-\Delta)^s w_R(x)\geq 0,& x\in \mathbb{R}^n,\\
\mathop{\lim}\limits_{|x| \to \infty} w_R(x) \geq 0.
\end{cases}
\end{eqnarray*}
By the maximum principle, we have
$$
w_R(x) \geq 0,\,\, x\in \mathbb{R}^n.
$$

Now let $R \to \infty$, we arrive at
$$
v(x)\geq \int_{\mathbb{R}^n}\frac{C_{n, s}}{|x-y|^{n-2s}} \left(K(\bar x) v^p(y)+b \right)dy, \,\, x \in \mathbb{R}^n.
$$
If $b>0,$ then the integral on the right hand side diverges and
$v(x) =\infty.$ This contradicts the boundedness of $v$.
Therefore, we must have $b=0$ and
\begin{eqnarray*}
(-\Delta)^s v(x)=K(\bar x) v^p(x),\,\, x\in \mathbb{R}^n.
\end{eqnarray*}
By virtue of \eqref{eq4-12} and  the maximum principle,
one gets that
$$v(x)>0,\,\, \forall x \in \mathbb{R}^n.$$

On the other hand, by the Liouville theorem in \cite[Theorem 4]{CLL2017},
\eqref{eq4-15}  has no positive solution if $1<p<\frac{n+2s}{n-2s}.$ This is a contradiction. Therefore, \eqref{eq1.8} and \eqref{eq1.9}
are true. Therefore,  the proof of Theorem \ref{th1.4} is completed.
\end{proof}

\begin{proof}[Proof of Theorem \ref{th1.5}.]
In this case,
 $\Omega$ is an unbounded domain in $\mathbb{R}^n$, and
 $u$ solves the  quasi-linear equation
$$
(-\Delta)^s u(x)=b(x) |\nabla u|^q(x)+ f(x, u(x)),\,\, x\in \Omega
$$
for $s>\frac{1}{2}.$

 We show also that
\begin{eqnarray}\label{eq4-16}
u(x)+|\nabla u|^{\frac{2s}{2s+p-1}}(x)\leq C\left(1+dist^{-\frac{2s}{p-1}}(x, \partial \Omega)\right),\,\, x\in \Omega
\end{eqnarray}
by a contradiction argument.

If \eqref{eq4-16} is not valid, then there exist sequences of $\Omega_k \subset \mathbb{R}^n,$ $u_k, \, x_k \in \Omega_k$
such that $u_k$ solves
\begin{eqnarray*}
(-\Delta)^s u_k(x)=b(x)|\nabla u_k|^q(x)+f(x, u_k(x)),& x\in \Omega_k,\\
\end{eqnarray*}
and
\begin{eqnarray*}
\frac{1}{2}M_k(x_k)d_k
\geq
\frac{1}{2}M_k(x_k) (1+d_k^{-1})^{-1}
>k\to \infty\,\, \mbox{as} \,\, k\to \infty,
\end{eqnarray*}
where $d_k=dist (x_k, \partial \Omega_k)$ and
$$
M_k(x):=u_k^{\frac{p-1}{2s}}(x)+|\nabla u_k|^{\frac{p-1}{2s+p-1}}(x).
$$

Define
$$
r_k:=2k M_k^{-1}(x_k).
$$
It follows that $r_k<d_k$ and therefore $B_{r_k}(x_k)\subset \Omega_k.$

To start the rescaling procedure and  obtain upper bounds for solutions, we consider the following function
$$
 S_k(x)=\left(M_k(x)(r_k-|x-x_k|)\right)^\frac{2s}{p-1},\,\, x\in B_{r_k}(x_k).
$$
Then there exists a point $a_k \in B_{r_k}(x_k)$  such that
$$
S_k(a_k)=\mathop{\max}\limits_{B_{r_k}(x_k)} S_k(x).
$$
It follows that
\begin{eqnarray}\label{eq4-17}
M_k(x)\leq M_k(a_k)\frac{r_k-|a_k-x_k|}{r_k-|x-x_k|},\,\, x\in B_{r_k}(x_k).
\end{eqnarray}

Denote
$$
\lambda_k=(M_k(a_k))^{-1}.
$$
Taking $x=x_k$ in \eqref{eq4-17}  gives
 \begin{eqnarray}\label{eq4-18}
\frac{M_k(x_k)}{M_k(a_k)}\leq \frac{r_k-|a_k-x_k|}{r_k},
\end{eqnarray}
which implies that
$$
M_k(x_k)\leq M_k(a_k).
$$
It follows from  the definition of $r_k$, $\lambda_k$ and  \eqref{eq4-18} that
 \begin{eqnarray}\label{eq4-19}
2k\lambda_k\leq r_k-|a_k-x_k|.
\end{eqnarray}

In addition, using \eqref{eq4-19} and by a similar argument as in deriving \eqref{eq4-6}, one gets
\begin{eqnarray}\label{eq4-20}
r_k-|a_k-x_k|\leq 2(r_k-|x-x_k|),\,\, x\in B_{k \lambda_k}(a_k)\subset B_{r_k}(x_k).
\end{eqnarray}
Combining \eqref{eq4-17} with \eqref{eq4-20} yields
\begin{eqnarray}\label{eq4-21}
	M_k(x)\leq 2 M_k(a_k) ,  \,\, x\in B_{k\lambda_k}(a_k).
\end{eqnarray}

Rescale the solutions sequence as
$$
v_k(x)=\frac{1}{(M_k(a_k))^{\frac{2s}{p-1}}}u_k(\lambda_k x+a_k),\,\, x\in B_k(0).
$$
Then $v_k(x)$ satisfies the following equation
\begin{eqnarray}\label{eq4.2}
\begin{aligned}
&(-\Delta)^s v_k(x)\\
=&\lambda_k^{\frac{2sp-(2s+p-1)q}{p-1}} b(\lambda_k x+a_k)|\nabla v_k|^q(x)
+\lambda_k^{\frac{2sp}{p-1}} f(\lambda_k x+a_k, \lambda_k^{-\frac{2s}{p-1}} v_k(x))\\
:=&F_k(x, v_k(x), \nabla v_k(x)),\,\, x\in B_k(0).
\end{aligned}
\end{eqnarray}

Moreover, by the definition of $\lambda_k$ and \eqref{eq4-21}, one has
\begin{eqnarray}\label{eq4.3}
\left(v_k^{\frac{p-1}{2s}}+|\nabla v_k|^{\frac{p-1}{2s+p-1}}\right)(0)=\lambda_k M_k(a_k)=1
\end{eqnarray}
and
\begin{eqnarray}\label{eq4.4}
\left(v_k^{\frac{p-1}{2s}}+|\nabla v_k|^{\frac{p-1}{2s+p-1}}\right)(x)\leq 2,\,\, |x|\leq k.
\end{eqnarray}
It follows from \eqref{eq4.4} that $v_k$ and $|\nabla v_k|$ are uniformly bounded in $B_k(0)$.
Then  for $k$ large enough,  one can get
$$
\lambda_k^{\frac{2sp-(2s+p-1)q}{p-1}} b(\lambda_k x+a_k)|\nabla v_k|^q(x)\leq C_8,\,\, x\in B_k(0),
$$
and
\begin{eqnarray*}
\begin{aligned}
&\lambda_k^{\frac{2sp}{p-1}} f(\lambda_k x+a_k, \lambda_k^{-\frac{2s}{p-1}} v_k(x))\\
\leq &C_0 \lambda_k^{\frac{2sp}{p-1}}\left(1+\lambda_k^{-\frac{2sp}{p-1}}(v_k(x))^p\right)\\
\leq &C_0 \lambda_k^{\frac{2sp}{p-1}}+C_0(v_k(x))^p\\
\leq & C_9,
\end{aligned}
\end{eqnarray*}
where one has used the assumptions (H1)-(H3)  and the fact $0<q<\frac{2sp}{2s+p-1}.$
As a result, there exists a constant $C_{10}>0$ independent of $k$ such that
$$
0\leq F_k(x, v_k(x), \nabla v_k(x))\leq C_{10},\,\, x\in B_k(0).
$$
 It follows
from Theorem \ref{th1.1} that there exists a positive constant $C_{11}$ independent of $k$ such that
$$
\|v_k\|_{C^{[2s-\varepsilon], \{2s-\varepsilon\}}(B_{3k/4}(0))}\leq C_{11}
$$
for all $0<\varepsilon <2s.$
Therefore  $F_k(x, v_k(x), \nabla v_k(x))$ in \eqref{eq4.2} is H\"{o}lder continuous for $x\in B_{3k/4}(0)$  due to $s>\frac{1}{2}.$  By virtue of Theorem \ref{th1.2}, there exist some  positive constants $\beta \in (0, 1)$ and $C_{12}$ which are independent of $k$ such that
\begin{eqnarray}\label{eq4.5}
\|v_k\|_{C^{[2s+\beta], \{2s+\beta\}}(B_{k/2}(0))}\leq C_{12}.
\end{eqnarray}

Due to the above uniform regularity estimates, it follows from the Ascoli-Arzel\'{a} theorem  that there exists a converging subsequence of $\{v_k\}$ (still denoted by $\{v_k\}$) that converges pointwisely in $\mathbb{R}^n$ to a function
$v(x)$, and
$$
v_k \to v \,\, \mbox{in}\,\, C_{{loc}}^1(\mathbb{R}^n) \,\,  \mbox{as}\,\, k\to \infty.
$$
Taking limit in \eqref{eq4.3}, we obtain
\begin{eqnarray}\label{eq4.6}
v^{\frac{p-1}{2s}}(0)+|\nabla v|^{\frac{p-1}{2s+p-1}}(0)= 1.
\end{eqnarray}
\eqref{eq4.5} and a similar argument as deriving \eqref{eq4-13} leads to
\begin{eqnarray}\label{eq4.7}
\mathop{\lim}\limits_{k\to \infty} (-\Delta)^s v_k(x)=(-\Delta)^s v(x) -  b,\,\, x\in \mathbb{R}^n
\end{eqnarray}
for some $b \geq 0$.

To derive a contradiction, one needs to analyze the limit equation satisfied by $v$.

Since $0<q<\frac{2sp}{2s+p-1}$, $v_k$ and $|\nabla v_k|$ are uniformly bounded in $B_k(0),$ then
$$
\lambda_k^{\frac{2sp-(2s+p-1)q}{p-1}} b(\lambda_k x+a_k)|\nabla v_k|^q(x)
\to 0\,\, \mbox{ as }\,\, k\to \infty,
$$

In addition, for any fixed $x\in \mathbb{R}^n,$ if $\{\lambda_k x+a_k\}$ is bounded, we assume that $\lambda_k x+a_k \to \bar x$  by extracting a further
subsequence. Therefore, the assumption (H3) implies
\begin{eqnarray}\label{eq4.8}
\begin{aligned}
&\lambda_k^{\frac{2sp}{p-1}} f(\lambda_k x+a_k, \lambda_k^{-\frac{2s}{p-1}} v_k(x))\\
= & (v_k(x))^p\frac{f(\lambda_k x+a_k, \lambda_k^{-\frac{2s}{p-1}} v_k(x))}{\left(\lambda_k^{-\frac{2s}{p-1}} v_k(x)\right)^p}\\
\to& K(\bar x) v^p(x)\,\, \mbox{ as }\,\, k\to \infty,
\end{aligned}
\end{eqnarray}
which is  still valid in the case $|\lambda_k x+a_k| \to \infty$ with $|\bar x|= \infty$ by (H3).

Consequently, combining \eqref{eq4.7} and \eqref{eq4.8} shows that $v$ is a solution of
\begin{eqnarray*}
(-\Delta)^s v(x)=K(\bar x) v^p(x) +  b,\,\, x\in \mathbb{R}^n.
\end{eqnarray*}
 Similar to the previous argument, one can show that $b=0$, and hence
\begin{eqnarray}\label{eq4.9}
(-\Delta)^s v(x)=K(\bar x) v^p(x),\,\, x\in \mathbb{R}^n.
\end{eqnarray}

By virtue of \eqref{eq4.6} and  the maximum principle,
one can deduce that
$$v(x)>0,\,\, \forall x \in \mathbb{R}^n.$$

On the other hand, by the Liouville theorem in \cite[Theorem 4]{CLL2017},
\eqref{eq4.9}  has no positive solution. This is a contradiction. Therefore, \eqref{eq4-16}
holds and thus  the proof of Theorem \ref{th1.5} is completed.
\end{proof}

{\bf{Acknowledgements.}} This research is supported in part by the Zhang Ge Ru Foundation, Hong Kong RGC Earmarked Research Grants CUHK-14301421, CUHK-14300917, CUHK-14302819 and CUHK-14300819. Z. Xin is also supported by the key
project of NSFC  (Grants No. 12131010 and 11931013). C. Li is partially supported by NSFC (Grant No. 12031012).  W. Chen is
 partially supported by MPS Simons Foundation 847690 and NSFC (Grant No. 12071229). L. Wu is  partially supported by NSFC (Grant No. 12401133).
 \medskip

{\bf{Date availability statement:}} Data will be made available on reasonable request.
\medskip

{\bf{Conflict of interest statement:}} There is no conflict of interest.

\end{document}